\documentclass{amsart}

\usepackage{graphicx}
\usepackage[shortlabels]{enumitem}

\usepackage{color}

\usepackage{tikz}
\usetikzlibrary{arrows,chains,matrix,positioning,scopes}
\usetikzlibrary{arrows.meta}
\usetikzlibrary{matrix}

\usepackage{hyperref}
\usepackage {cleveref}

\newtheorem{theorem}{Theorem}[section]
\newtheorem{proposition}[theorem]{Proposition}
\newtheorem{lemma}[theorem]{Lemma}
\newtheorem{corollary}[theorem]{Corollary}

\theoremstyle{definition}
\newtheorem{definition}[theorem]{Definition}

\theoremstyle{remark}
\newtheorem{remark}[theorem]{Remark}
\newtheorem{example}[theorem]{Example}

\numberwithin{equation}{section}

\begin{document}

\title{Bordism of Constrained Morse Functions}

\author[D.J. Wrazidlo]{Dominik J. Wrazidlo}

\address{Institute of Mathematics for Industry, Kyushu University, Motooka 744, Nishi-ku, Fukuoka 819-0395, Japan}
\email{d-wrazidlo@math.kyushu-u.ac.jp}


\date{\today.}


\begin{abstract}
We call a Morse function $f$ on a closed manifold \emph{$k$-constrained} if neither $f$ nor $-f$ has critical points of indefinite Morse index $< k$.
In this paper we study bordism groups of $k$-constrained Morse functions, and thus interpolate between the case $k = 1$ of bordism groups of Morse functions (computed by Ikegami) and the case $k \gg 1$ of bordism groups of special generic functions (computed by Saeki).

We employ Levine's elimination of cusps, Stein factorization, the two-index theorem of Hatcher-Wagoner, and a handle extension theorem for fold maps due to Gay-Kirby to show that the notion constrained bordism is strongly related to so-called connective bordism.
As an application of our results we show that the oriented bordism group of constrained Morse functions detects exotic Kervaire spheres in certain dimensions.

\end{abstract}

\maketitle
\tableofcontents

\section{Introduction}

Bordism theory for differentiable maps has been a central issue in the field of global singularity theory ever since it was initiated in the mid-1950s by Ren\'{e} Thom \cite{thom}.
Thom computed bordism groups of embedded manifolds, and his homotopy theoretic approach relies crucially on the Pontrjagin-Thom construction.
Later, Rim\'{a}nyi and Sz\H{u}cs \cite{rim} adopted this concept to study bordism groups of smooth maps with certain prescribed types of singularities, where they considered maps of $m$-manifolds into a fixed $n$-manifold in the case of positive codimension, $n-m > 0$.
Kalm\'{a}r \cite{kal} found an analogous construction for smooth maps with prescribed singular fibers in the case of negative codimension.

More explicit methods of geometric topology have recently been applied successfully to compute bordism groups of smooth maps with concrete singularities of the mildest types.
For instance, bordism groups of Morse functions, which were originally introduced by Ikegami and Saeki in \cite{ikesae}, have been determined entirely by Ikegami \cite{ike} by using Levine's cusp elimination technique \cite{lev} and the Kervaire semi-characteristic \cite{ker}.
Furthermore, employing the technique of Stein factorization \cite{burder} as well as Cerf's pseudo-isotopy theorem \cite{cer}, Saeki \cite{sae2} showed that oriented bordism groups of so-called special generic functions, i.e., Morse functions having only minima and maxima as critical points, are isomorphic to groups of homotopy spheres \cite{KM}.
More generally, Sadykov \cite{sad} combined the Pontrjagin-Thom construction with Smale-Hirsch theory \cite{hirsmale} to express bordism groups of special generic maps in terms of stable homotopy theory.
\par\medskip

In this paper, we develop a geometric-topological approach to study
bordism groups of Morse functions whose critical points are subject to the following type of index constraints.
For a given integer $k \geq 1$ we call a Morse function on a closed $n$-manifold \emph{$k$-constrained} if all indefinite Morse indices of its critical points are contained in the interval $\{k, \dots, n-k\}$.
Thus, the notion of a $k$-constrained Morse function interpolates between ordinary Morse functions ($k = 1$) and special generic functions ($k > n/2$).
From the viewpoint of Morse theory \cite{mil2} we make the fundamental observation that a closed manifold of dimension $n \neq 4$ admits a $k$-constrained Morse function if and only if it is $(k-1)$-connected (where the case $n = 3$ relies on Perelman's solution to the smooth Poincar\'{e} conjecture).
This observation suggests that bordism groups of constrained Morse functions should be related strongly to so-called connective bordism groups (see \Cref{connective bordism}), which will in fact be manifest in our \Cref{theorem constrained generic bordism and connective bordism}.
In the context of a generalization of the Madsen-Weiss theorem, Perlmutter \cite{perl} has recently imposed Morse index constraints of the above form on Morse functions defined on morphisms of the bordism category.

The central notion of $k$-constrained bordism (see \Cref{definition constrained bordism}) involves so-called fold maps of bordisms into the plane.
By definition, a fold map on a manifold of dimension $n+1$ is a smooth map all of whose singular points are determined by map germs of the form
\begin{displaymath}
(t, x_{1}, \dots, x_{n}) \mapsto (t, -x_{1}^{2} - \dots - x_{i}^{2} + x_{i+1}^{2} + \dots + x_{n}^{2}).
\end{displaymath}
Thus, fold maps into the plane can locally be thought of as one-parameter families of Morse functions, and the \emph{absolute index} $\operatorname{max}\{i, n-i\}$ of a fold point with the above map germ turns out to be a locally constant invariant of singular points.
Depending on whether the absolute index of a component of the singular locus is 
Therefore, index constraints imposed on Morse functions induce constraints on the absolute index of fold maps in a natural way.
In generalization of the notion of (in)definite Morse critical points, fold points of index $n$ are called \emph{definite}, and \emph{indefinite} otherwise.
In \Cref{definition constrained bordism} we will introduce the correct notion of bordism between constrained Morse functions.

Kitazawa's construction of so-called round fold maps \cite[p. 339]{kita} implies that total spaces of fiber bundles over spheres with fiber a twisted sphere admit fold maps into the plane with connected indefinite fold locus.
The general existence problem for fold maps in the presence of index constraints has been posed by Saeki in Problem 5.13 of \cite[p. 200]{sae3}.
Technical difficulties in approaching this problem arise from the fact that Eliashberg's $h$-principle \cite{eli} cannot be used when constructing fold maps which are exposed to index constraints.
The results of this paper can be considered as a partial solution to Saeki's problem for the case of fold maps into $\mathbb{R}^{2}$ that are subject to our type of index constraints.

\par\medskip
In the following discussion of our main results the focus lies on oriented bordism groups; unoriented versions of the results hold in an analogous way, and some details are pointed out in \Cref{remark unoriented version for constrained generic bordism and connective bordism} and \Cref{remark unoriented version for bordism of constrained Morse functions}.

Besides the oriented $n$-dimensional bordism group of $k$-constrained Morse functions, which will be denoted by $\mathcal{M}_{k}^{n}$ (see \Cref{definition constrained bordism}), our two main results below involve the oriented $k$-connective $n$-dimensional bordism group $\mathcal{C}_{k}^{n}$ as reviewed in \Cref{connective bordism}, as well as the group $\mathcal{G}_{k}^{n}$ of oriented $k$-constrained generic $n$-dimensional bordism (see \Cref{definition constrained generic bordism}).
These two bordism groups both interpolate between the smooth oriented bordism group and the group of homotopy spheres (see \Cref{remark interpolation of generic constrained bordism}).

In our first main result we extend Saeki's result \cite{sae2} on bordism groups of special generic functions to constrained Morse functions by relating groups of constrained generic bordism to connective bordism groups as follows.

\begin{theorem}\label{theorem constrained generic bordism and connective bordism}
Let $n \geq 6$ and $1 \leq k < n$ be integers.
Then, there exist homomorphisms as follows:
\begin{enumerate}[(i)]
\item $\varepsilon_{k}^{n} \colon \mathcal{C}_{k}^{n} \rightarrow \mathcal{G}_{k}^{n}$, $[M^{n}] \mapsto [f]$, where $f \colon M^{n} \rightarrow \mathbb{R}$ denotes an arbitrarily chosen $k$-constrained Morse function, and
\item for $k > 1$, $\delta_{k}^{n} \colon \mathcal{G}_{k}^{n} \rightarrow \mathcal{C}_{k-1}^{n}$, $[f \colon M^{n} \rightarrow \mathbb{R}] \mapsto [\sharp(M^{n})]$, where $\sharp(M^{n})$ denotes the oriented connected sum of the connected components of $M^{n}$. (We use the convention that $\sharp (\emptyset) = S^{n}$.)
\end{enumerate}
Moreover, for $1 < k < n$, the natural homomorphism $\mathcal{C}_{k}^{n} \rightarrow \mathcal{C}_{k-1}^{n}$, $[M^{n}] \mapsto [M^{n}]$, factors as the composition $\delta^{n}_{k} \circ \varepsilon^{n}_{k}$,
and the natural homomorphism $\mathcal{G}_{k}^{n} \rightarrow \mathcal{G}_{k-1}^{n}$, $[f \colon M^{n} \rightarrow \mathbb{R}] \mapsto [f]$, factors as the composition $\varepsilon^{n}_{k-1} \circ \delta^{n}_{k}$.
\end{theorem}

The two-index theorem of Hatcher and Wagoner \cite{hatwag} (see \Cref{two-index theorem}) will serve as an essential tool for showing that the homomorphism of part $(i)$ is well-defined
in the case $1 < k < n/2$ (see the proof of \Cref{existence of constrained bordisms}).
Furthermore, we will exploit a handle extension theorem for constrained Morse functions that has recently been established by Gay and Kirby \cite{gaykir} in the context of symplectic geometry (see \Cref{handle extension theorem}).
In order to show that the homomorphism of part $(ii)$ is well-defined, we use Stein factorization for generic maps into the plane which are subject to certain fold index constraints (see \Cref{stein factorization}).

Our second main result (see \Cref{theorem computation of oriented bordism of constrained Morse functions} below) reveals that bordism groups of constrained Morse functions have a structure similar to that of bordism groups of Morse functions \cite{ike}, whereas a somewhat surprising phenomenon arises in dimensions of the form $n \equiv 3 \, \operatorname{mod} 4$ (see parts $(iii)$ and $(iv)$).
Namely, the size of the group $\mathcal{M}_{k}^{n}$ is governed by an integer $\kappa_{(n+1)/4}$ that measures the existence of closed $k$-constrained bordisms of dimension $n+1$ with odd Euler characteristic (see \Cref{definition sequence existence of odd euler characteristic and constrained generic map}).
Furthermore, the techniques that are used in the proof of \Cref{theorem constrained generic bordism and connective bordism} allow us to relate the sequence $\kappa_{1}, \kappa_{2}, \dots$ to another sequence $\gamma_{1}, \gamma_{2}, \dots$ of positive integers measuring the existence of highly-connected closed manifolds with odd Euler characteristic (see \Cref{definition sequence gamma}).

\begin{theorem}\label{theorem computation of oriented bordism of constrained Morse functions}
Let $n \geq 4$ and $1 < k \leq n/2$ be integers.
The oriented $n$-dimensional bordism group of $k$-constrained Morse functions $\mathcal{M}_{k}^{n}$ fits into a short exact sequence of abelian groups
\begin{displaymath}
0 \rightarrow A^{n}_{k} \stackrel{\alpha_{k}^{n}}{\longrightarrow} \mathcal{M}_{k}^{n} \stackrel{\beta_{k}^{n}}{\longrightarrow} \mathcal{G}_{k}^{n} \oplus \mathbb{Z}^{\oplus \lfloor n/2 \rfloor - k} \rightarrow 0,
\end{displaymath}
where the homomorphisms $\alpha_{k}^{n}$ and $\beta_{k}^{n}$ are defined in \Cref{lemma definition of homomorphism alpha} and \Cref{lemma surjectivity of beta}, respectively.
We have either $A^{n}_{k} = 0$ or $A^{n}_{k} = \mathbb{Z}/2$, depending on the following cases:
\begin{enumerate}[$(i)$]
\item If $n$ is even, then $A^{n}_{k} = 0$, so $\beta_{k}^{n}$ is an isomorphism.
\item If $n \equiv 1 \, \operatorname{mod} 4$, then $A^{n}_{k} = \mathbb{Z}/2$, and $\alpha_{k}^{n}$ admits a splitting (see \Cref{existence of a splitting}).
\item If $n \equiv 3 \, \operatorname{mod} 4$ and $k \leq \kappa_{(n+1)/4}$, then $A^{n}_{k} = 0$, so $\beta_{k}^{n}$ is an isomorphism.
\item If $n \equiv 3 \, \operatorname{mod} 4$ and $k > \kappa_{(n+1)/4}$, then $A_{k}^{n} = \mathbb{Z}/2$.
\end{enumerate}
Furthermore, the constants $\kappa_{(n+1)/4}$ and $\gamma_{(n+1)/4}$ (see \Cref{definition sequence existence of odd euler characteristic and constrained generic map} and \Cref{definition sequence gamma}, respectively) are related by $\gamma_{(n+1)/4} \leq \kappa_{(n+1)/4} \leq \gamma_{(n+1)/4} + 1$.
\end{theorem}

As pointed out in Remark 3.6 of \cite[p. 296]{sae2}, index constraints for fold singularities seem to be important for understanding global differential topological phenomena like smooth structures on manifolds.
As a new instance of this principle we apply our results in \Cref{application to Kervaire spheres} to show how bordism groups of constrained Morse functions can detect individual exotic smooth structures on spheres.
More specifically, \Cref{theorem kervaire spheres} states that, in infinitely many dimensions of the form $n \equiv 1 \, \operatorname{mod} 4$, the exotic Kervaire $n$-sphere can be characterized among all exotic $n$-spheres by the property that it admits a $(n-1)/2$-constrained Morse function representing $0 \in \mathcal{M}_{k}^{n}$.
Results of this type are relevant in the context of a concrete positive TFT which has recently been constructed by Banagl (see Section 10 of \cite{ban}), and in \cite{wra} we use them to compute Banagl's aggregate invariant of various exotic spheres.

\subsection*{Notation}
All manifolds and maps considered in this paper are differentiable of class $C^{\infty}$.
If $M^{n}$ is a manifold with boundary, then we write $\operatorname{int} M^{n} = M^{n} \setminus \partial M^{n}$ for its interior.
For an oriented manifold $M^{n}$ the manifold equipped with the opposite orientation will be denoted by $-M^{n}$.
Writing $||x||^{2} := x_{1}^{2} + \dots + x_{n}^{2}$ for $x = (x_{1}, \dots, x_{n}) \in \mathbb{R}^{n}$, we denote the unit $n$-disc of radius $r$ by $D_{r}^{n} = \{x \in \mathbb{R}^{n}; \; ||x|| \leq r\}$.
The symbol $\cong$ will either mean orientation preserving diffeomorphism of manifolds or isomorphism of groups.

\subsection*{Acknowledgements}
Many parts of the present paper originate from the author's Heidelberg PhD thesis, and the author would like to express his deep gratitude to his supervisor Professor Markus Banagl for inspiring guidance during the creation of this work.
Moreover, the author would like to thank Professor Osamu Saeki for helpful discussions.

The author is grateful to the German National Merit Foundation (Studienstiftung des deutschen Volkes) for financial support.
Moreover, the author has been supported by JSPS KAKENHI Grant Number JP17H06128.

\section{Connective Bordism}\label{connective bordism}

\Cref{theorem constrained generic bordism and connective bordism} relates constrained generic bordism groups (see \Cref{definition constrained generic bordism}) to connective bordism groups defined as follows (compare \cite[Example 17, p. 51]{sto}).

\begin{definition}\label{definition connective bordism}
Fix integers $n \geq 2$ and $1 \leq k < n$.
Let $M^{n}$ and $N^{n}$ be non-empty $k$-connected oriented closed $n$-manifolds.
An \emph{oriented $k$-connective bordism from $M^{n}$ to $N^{n}$} is a $k$-connected oriented compact manifold $W^{n+1}$ with boundary $\partial W^{n+1} = M^{n} \sqcup -N^{n}$.

The \emph{oriented $k$-connective $n$-bordism group $\mathcal{C}_{k}^{n}$} is the set of equivalence classes $[M^{n}]$ of non-empty $k$-connected oriented closed $n$-manifolds $M^{n}$ subject to the equivalence relation of oriented $k$-connective bordism.
\end{definition}

It can be shown that $\mathcal{C}_{k}^{n}$ is for any $1 \leq k < n$ an abelian group with group law induced by oriented connected sum, $[M^{n}] + [N^{n}] := [\sharp(M^{n} \sqcup N^{n})]$, identity element represented by the standard sphere $S^{n}$, and inverses induced by reversing the orientation, $- [M^{n}] = [-M^{n}]$.
Moreover, using the characterization of $h$-cobordisms given in \cite[p. 108]{mil2}, one can show that $\mathcal{C}_{k}^{n}$ coincides for $k > (n-1)/2$ with the group of homotopy spheres $\Theta_{n}$ as defined in \cite{KM}.

\begin{proposition}\label{proposition bott periodicity}
Let $1 < k \leq (n-1)/2$.
The natural homomorphism $\mathcal{C}_{k}^{n} \rightarrow \mathcal{C}_{k-1}^{n}$, $[M^{n}] \mapsto [M^{n}]$, is injective for $k \equiv 3, 5, 6, 7 \, \operatorname{mod} 8$.
\end{proposition}

\begin{proof}
If $M^{n}$ represents $0 \in \mathcal{C}_{k-1}^{n}$, then there exists an oriented $(k-1)$-connected compact manifold $W^{n+1}$ with boundary $\partial W^{n+1} = M^{n}$.
Note that any triangulation of $W^{n+1}$ is $(k-1)$-parallelizable, that is, $TW$ is trivial over the $(k-1)$-skeleton (see \cite[Section 5, p. 49]{mil}).
The obstruction for being $k$-parallelizable vanishes since $\pi_{k-1}(SO(n)) = 0$ for $k \equiv 3, 5, 6, 7 \, \operatorname{mod} 8$ (see the proof of \cite[Theorem 3.1, p. 508]{KM}).
Therefore, by \cite[Theorem 3, p. 49]{mil} $W$ can be made $\operatorname{min}\{k, \lfloor \operatorname{dim}W^{n+1}/2 - 1\rfloor\} = k$-connected by a finite sequence of surgeries without changing $M^{n} = \partial W^{n+1}$.
Hence, if $M^{n}$ happens to be $k$-connected, then $M^{n}$ represents $0 \in \mathcal{C}_{k}^{n}$.
\end{proof}

\par\medskip

Poincar\'{e} duality implies that orientable closed manifolds with odd Euler characteristic can only exist in dimensions which are a multiple of $4$.
For instance, $\mathbb{C}P^{2i}$ is for any integer $i \geq 1$ a simply connected closed $4i$-manifold with odd Euler characteristic.
We define a sequence $\gamma_{1}, \gamma_{2}, \dots$ of positive integers as follows (compare Problem 2.6 in \cite[p. 151]{dovschu}).

\begin{definition}\label{definition sequence gamma}
For every integer $i \geq 1$ let $\gamma_{i}$ be the greatest integer $k \geq 1$ for which there exists a $k$-connected closed manifold $V^{4i}$ with odd Euler characteristic (or, equivalently, odd signature).
\end{definition}

By an argument analogous to the proof of \Cref{proposition bott periodicity}, we can show that $\gamma_{i} \not\equiv 2, 4, 5, 6 \, \operatorname{mod} 8$ for all $i \geq 1$.

\begin{example}
Note that $\gamma_{i} < 2i$ because $2i$-connected closed $4i$-manifolds are homotopy spheres.
For odd $i$, we always have $\gamma_{i} = 1$ because $2$-connected closed $4i$-manifolds $V^{4i}$ are spinable, which implies that their signature is a multiple of $16$ according to Ochanine's generalization of Rochlin's theorem (see \cite[p. 133]{och}).
For even $i$, we have $\gamma_{i} \geq 3$ because the quaternionic projective space $\mathbb{H}P^{i}$ has odd Euler characteristic.
In particular, $\gamma_{2} = 3$.
When $i = 4j$ is a multiple of $4$, then $\gamma_{i} \geq 7$ because the $j$-fold product $\mathbb{O}P^{2} \times \dots \times \mathbb{O}P^{2}$ of the octonionic projective plane $\mathbb{O}P^{2}$ is a $7$-connected closed manifold with odd Euler characteristic.
In particular, $\gamma_{4} = 7$.
\end{example}

\section{Preliminaries on Generic Maps into the Plane}

In this section we collect essential techniques for constructing and studying generic maps from bordisms into the plane.

Fix an integer $n \geq 1$.
Recall that any smooth map of a manifold $X^{n+1}$ into the plane can be approximated arbitrarily well in the Whitney $C^{\infty}$ topology by a smooth map $G \colon X^{n+1} \rightarrow \mathbb{R}^{2}$ whose singular locus $S(G) = \{x \in X^{n+1}; \; \operatorname{rank} d_{x}G < 2\}$ consists of those $x \in X^{n+1}$ admitting coordinate charts centered at $x$ and $G(x)$, respectively, in which $G$ has one of the following normal forms:
\begin{enumerate}[(1)]
\item $(t, x_{1}, \dots, x_{n}) \mapsto (t, t x_{1} + x_{1}^3 \pm x_{2}^{2} \pm \dots \pm x_{n}^{2})$, i.e., $x$ is a \emph{cusp} of $G$.
\item $(t, x_{1}, \dots, x_{n}) \mapsto (t, \pm x_{1}^{2} \pm \dots \pm x_{n}^{2})$, i.e., $x$ is a \emph{fold point} of $G$.
\end{enumerate}

\begin{definition}\label{definition constrained generic bordism}
Let $f \colon M^{n} \rightarrow \mathbb{R}$ and $g \colon N^{n} \rightarrow \mathbb{R}$ be $k$-constrained Morse functions on oriented closed $n$-manifolds.
An \emph{oriented $k$-constrained generic bordism from $f$ to $g$} is an oriented bordism $W^{n+1}$ from $M^{n}$ to $-N^{n}$ equipped with a generic map $G \colon W^{n+1} \rightarrow \mathbb{R}^{2}$ such that
\begin{enumerate}[$(i)$]
\item there exist tubular neighborhoods $M^{n} \times [0, \varepsilon) \subset W^{n+1}$ of $M^{n} \times \{0\} = M^{n} \subset W^{n+1}$ and $N^{n} \times (1-\varepsilon, 1] \subset W^{n+1}$ of $N^{n} \times \{1\} = N^{n} \subset W^{n+1}$ such that
\begin{displaymath}
G|_{M^{n} \times [0, \varepsilon)} = f \times \operatorname{id}_{[0, \varepsilon)}, \qquad G|_{N^{n} \times (1-\varepsilon, 1]} = g \times \operatorname{id}_{(1-\varepsilon, 1]}.
\end{displaymath}
\item all absolute indices of fold points of $G$ are contained in $\{\lceil n/2\rceil, \dots, n-k\} \cup \{n\}$.
\end{enumerate}

The \emph{oriented $k$-constrained generic $n$-bordism group $\mathcal{G}_{k}^{n}$} is the set of equivalence classes $[f]$ of $k$-constrained Morse functions $f \colon M^{n} \rightarrow \mathbb{R}$ on oriented closed $n$-manifolds subject to the equivalence relation of oriented $k$-constrained generic bordism.
\end{definition}

\begin{definition}\label{definition constrained bordism}
Let $f \colon M^{n} \rightarrow \mathbb{R}$ and $g \colon N^{n} \rightarrow \mathbb{R}$ be $k$-constrained Morse functions on oriented closed $n$-manifolds.
An \emph{oriented $k$-constrained bordism from $f$ to $g$} is an oriented $k$-constrained generic bordism from $f$ to $g$ without cusps.

The \emph{oriented $n$-bordism group of $k$-constrained Morse functions $\mathcal{M}_{k}^{n}$} is the set of equivalence classes $[f]$ of $k$-constrained Morse functions $f \colon M^{n} \rightarrow \mathbb{R}$ on oriented closed $n$-manifolds subject to the equivalence relation of oriented $k$-constrained bordism.
\end{definition}

Note that $\mathcal{G}_{k}^{n}$ and $\mathcal{M}_{k}^{n}$ are abelian groups.
In both cases, the group law is induced by disjoint union, $[f \colon M^{n} \rightarrow \mathbb{R}] + [g \colon N^{n} \rightarrow \mathbb{R}] := [f \sqcup g \colon M^{n} \sqcup N^{n} \rightarrow \mathbb{R}]$, the identity element is represented by the unique map $f_{\emptyset} \colon \emptyset \rightarrow \mathbb{R}$, and the inverse of $[f \colon M^{n} \rightarrow \mathbb{R}]$ is given by $[-f \colon -M^{n} \rightarrow \mathbb{R}]$. \par\medskip

There are natural homomorphisms $\mathcal{G}_{l}^{n} \rightarrow \mathcal{G}_{k}^{n}$ and $\mathcal{M}_{l}^{n} \rightarrow \mathcal{M}_{k}^{n}$ whenever $l \geq k$.
Moreover, there is a natural homomorphism $\mathcal{M}_{k}^{n} \rightarrow \mathcal{G}_{k}^{n}$ which maps the class $[f \colon M^{n} \rightarrow \mathbb{R}] \in \mathcal{M}_{k}^{n}$ to the class $[f \colon M^{n} \rightarrow \mathbb{R}] \in \mathcal{G}_{k}^{n}$.

\begin{remark}\label{G and M coincide for large k}
By definition, the groups $\mathcal{G}_{k}^{n}$ and $\mathcal{M}_{k}^{n}$ coincide for $k > n/2$ both with the oriented bordism group of special generic functions on $n$-manifolds $\widetilde{\Gamma}(n, 1)$ as defined in \cite{sae2}.
\end{remark}

\begin{remark}\label{remark interpolation of generic constrained bordism}
Varying $k$, the group $\mathcal{G}_{k}^{n}$ interpolates between the smooth oriented bordism group $\Omega_{n}^{SO}$ (an isomorphism $\mathcal{G}_{1}^{n} \stackrel{\cong}{\longrightarrow} \Omega_{n}^{SO}$ is given by $[f \colon M^{n} \rightarrow \mathbb{R}] \mapsto [M^{n}]$) and, by \cite{sae2}, the group of homotopy spheres $\Theta_{n} \cong \widetilde{\Gamma}(n, 1) = \mathcal{G}_{k}^{n}$ ($k > n/2$).
\end{remark}

\subsection{Elimination of Cusps; Cusps and Euler Characteristic}\label{section cusps}
We refer to \cite{ike} for a detailed discussion of the material presented in this section.

Recall from \cite[Definition 2.2, p. 213]{ike} that there exist homomorphisms
\begin{displaymath}
\widetilde{\varphi}_{\lambda} \colon \mathcal{M}_{1}^{n} \rightarrow \mathbb{Z}, \quad [f] \mapsto C_{\lambda}(f) - C_{n-\lambda}(f), \qquad \lambda \in \{0, \dots, n\},
\end{displaymath}
where $C_{\mu}(f)$ denotes the number of critical points of $f$ of Morse index $\mu$.
For any integer $1 < k \leq n/2$ we use the natural homomorphism $\mathcal{M}_{k}^{n} \rightarrow \mathcal{M}_{1}^{n}$ to define a homomorphism (compare \cite[Definition 2.3, p. 213]{ike})
\begin{displaymath}
\Phi_{k}^{n} \colon \mathcal{M}_{k}^{n} \rightarrow \mathbb{Z}^{\oplus \lfloor n/2 \rfloor - k}, \quad [f] \mapsto (\widetilde{\varphi}_{\lfloor (n+3)/2\rfloor}([f]), \dots, \widetilde{\varphi}_{n-k}([f])).
\end{displaymath}

Levine's technique \cite{lev} for eliminating pairs of cusps of generic maps into the plane (see also \cite[Section 3, pp. 215ff]{ike}) can be used as in \cite{ike} to prove the following

\begin{theorem}\label{proposition creating and eliminating cusps}
Suppose that $n \geq 2$.
Let $G \colon W^{n+1} \rightarrow \mathbb{R}^{2}$ be an oriented $k$-constrained generic bordism from $g_{0} \colon M_{0}^{n} \rightarrow \mathbb{R}$ to $g_{1} \colon M_{1}^{n} \rightarrow \mathbb{R}$.
Suppose that $\Phi_{k}^{n}([g_{0}]) = \Phi_{k}^{n}([g_{1}])$.
Moreover, if $n$ is odd, then suppose that $G$ has an even number of cusps.
Then, $[g_{0}] = [g_{1}] \in \mathcal{M}_{k}^{n}$.
\end{theorem}

\begin{proof}
We make $W^{n+1}$ connected by using the oriented connected sum operation, and modify $G$ accordingly while performing the oriented connected sum along small $2$-discs centered at definite fold points of $G$.
If $W^{n+1}$ is connected, then $G$ is homotopic rel $\partial W^{n+1}$ to an oriented $k$-constrained bordism from $g_{0}$ to $g_{1}$ by means of an iterated elimination of matching pairs of cusps.
For details, see \cite[proof of Theorem 2.7, p. 220ff]{ike}.
\end{proof}

\begin{remark}\label{remark no cusps for 2k = n}
For $k = n/2 > 1$ it can be shown that any oriented $n/2$-constrained generic bordism $G \colon W^{n+1} \rightarrow \mathbb{R}^{2}$ is already an oriented $n/2$-constrained bordism.
Indeed, the map $G$ cannot have cusps because the occuring absolute fold indices $n$ and $n/2$ are not consecutive integers when $n/2 > 1$.
Consequently, $\mathcal{M}_{n/2}^{n} = \mathcal{G}_{n/2}^{n}$.
\end{remark}

By an adaption of the proof of \cite[Lemma 5.2, p. 226]{ike} we have the following

\begin{proposition}\label{proposition euler characteristic and cusps}
Let $G \colon W^{n+1} \rightarrow \mathbb{R}^{2}$ be an oriented $1$-constrained generic bordism from $g_{0} \colon M_{0}^{n} \rightarrow \mathbb{R}$ to $g_{1} \colon M_{1}^{n} \rightarrow \mathbb{R}$.
Let $c$ denote the number of cusps of $G$, and let $\nu$ denote the number of critical points of $g_{0} \sqcup g_{1}$.
Then, $\nu$ is even, and $c + \nu/2 \equiv \chi(W^{n+1}) \, \operatorname{mod} 2$, where $\chi(W^{n+1})$ denotes the Euler characteristic of $W^{n+1}$.
\end{proposition}

\subsection{Two-Index Theorem}\label{two-index theorem}

The purpose of this section is to discuss the two-index theorem of Hatcher and Wagoner \cite{hatwag}.
This theorem is based on a parametrized implementation of the Smale trick, by which one may trade a Morse critical point of index $i$ for one of index $i+2$ by creating a pair of critical points of successive indices $i+1$ and $i+2$, and then cancelling the Morse critical point of index $i$ with that of index $i+1$.
Under stronger assumptions the Smale trick has been used by Cerf in his proof of the pseudo-isotopy theorem (see \cite[Lemma 0, p. 101]{cer}).

\begin{theorem}\label{theorem cylinder bordism between constrained morse functions}
Fix integers $n \geq 5$ and $1 < k < n/2$.
Suppose that $f_{0}, f_{1} \colon M^{n} \rightarrow \mathbb{R}$ are $k$-constrained Morse functions on a closed manifold $M^{n}$.
Then, there exists an oriented $k$-constrained generic bordism $F \colon M^{n} \times [0, 1] \rightarrow \mathbb{R}^{2}$ from $f_{0}$ to $f_{1}$.
\end{theorem}

\begin{proof}
Without loss of generality, we may assume that $M^{n}$ is connected, and that $f_{0}(M^{n}) = f_{1}(M^{n}) = [0, 1]$.
For $i = 0, 1$ let $c_{i}^{0}$ and $c_{i}^{1}$ denote the unique critical points of $f_{i}$ of index $0$ and $n$, respectively.
For $i, j \in \{0, 1\}$ and suitable $\varepsilon > 0$ there exist orientation preserving embeddings $\iota_{i}^{j} \colon D^{n}_{2 \varepsilon} \rightarrow M$ such that $\iota_{i}^{j}(0) = c_{i}^{j}$ and
\begin{displaymath}
(f_{i} \circ \iota_{i}^{j})(x) = e^{j}(||x||^{2}) := j + (-1)^{j}||x||^{2}, \qquad x \in D^{n}_{2 \varepsilon}. \qquad (\ast)
\end{displaymath}
Furthermore, for possibly smaller $\varepsilon > 0$, there exists an isotopy $H \colon [0, 1] \times M \rightarrow M$ of diffeomorphisms $H_{t} := H(t, -) \colon M \rightarrow M$ such that $H_{0} = \operatorname{id}_{M}$ and $H_{1} \circ \iota_{0}^{j} = \iota_{1}^{j}$ for $j = 0, 1$.
Therefore, after replacing $f_{1}$ by $f_{1} \circ H_{1}$, we may without loss of generality work with the assumption that $f_{0} \circ \iota_{0}^{j} = f_{1} \circ \iota_{0}^{j}$ for $j = 0, 1$.
Set $U^{j} := \iota_{0}^{j}(D^{n}_{\varepsilon})$ for $j = 0, 1$.
Set $V := M \setminus (\iota_{0}^{0}(\operatorname{int} D^{n}_{\varepsilon}) \cup \iota_{0}^{1}(\operatorname{int} D^{n}_{\varepsilon}))$ and $V^{j} := U^{j} \cap V \cong S^{n-1}$ for $j = 0, 1$.
Then, $f_{i}$ restricts for $i = 0, 1$ to a Morse function
\begin{displaymath}
g_{i} := f_{i}|_{V} \colon (V, V^{0}, V^{1}) \rightarrow ([\varepsilon^{2}, 1-\varepsilon^{2}], \varepsilon^{2}, 1-\varepsilon^{2})
\end{displaymath}
all of whose critical points have Morse index contained in the set $\{k, \dots, n-k\}$.
Choose a generic $1$-parameter family $g_{t}$, $t \in [0, 1]$, as described in \cite[Theorem 9.4, pp. 188f]{cieeli} with regular level sets $V_{0} = g_{t}^{-1}(\varepsilon^{2})$ and $V_{1} = g_{t}^{-1}(1-\varepsilon^{2})$.
Since the cardinality of the set $\{k, \dots, n-k\}$ is at least $2$ (recall that $k < n/2$), we can use the two-index theorem of Hatcher and Wagoner (see \cite[Chapter V, Proposition 3.5]{hatwag}) in the form presented in \cite[Section 9.9, pp. 212f]{cieeli} to modify the family $g_{t}$ rel $g_{0}$ and $g_{1}$ iteratively in such a way that the resulting generic map $G \colon V \times [0, 1] \rightarrow [\varepsilon^{2}, 1-\varepsilon^{2}] \times [0, 1]$, $(x, t) \mapsto (g_{t}(x), t)$ is $k$-constrained.

In the following, we sketch the construction of the desired map $F$, which amounts to a careful extension of $G$ over $U^{j} \times [0, 1]$ for $j = 0, 1$.
(The construction is presented in full detail in \cite[Section 8.4]{wra} using \cite[Appendix B]{wra}.)
Without loss of generality, we may assume for $t \in [0, 1]$ that $g_{t} = g_{0}$ when $t$ is near $0$, and that $g_{t} = g_{1}$ when $t$ is near $1$.
We extend $g \colon V \times [0, 1] \rightarrow [\varepsilon^{2}, 1-\varepsilon^{2}]$ to a smooth map $\tilde{g} \colon \tilde{V} \times [0, 1] \rightarrow \mathbb{R}$ for some open neighborhood $\tilde{V}$ of $V$ in $M$ such that,  for $t \in [0, 1]$, $\tilde{g}|_{\tilde{V} \times \{t\}} = f_{0}|_{\tilde{V}}$ when $t$ is near $0$, and that $\tilde{g}|_{\tilde{V} \times \{1\}} = f_{1}|_{\tilde{V}}$ when $t$ is near $1$.
For $j = 0, 1$, we define a tubular neighborhood of $V^{j} \times [0, 1]$ in $\tilde{V} \times [0, 1]$ by
\begin{displaymath}
\alpha^{j} \colon (-\delta, \delta) \times V^{j} \times [0, 1] \rightarrow \tilde{V} \times [0, 1], \qquad \alpha^{j}(u, v, t) = (\iota_{0}^{j}(\rho(u) \cdot (\iota_{0}^{j})^{-1}(v)), t),
\end{displaymath}
where $\rho \colon (-1/2, 1/2) \rightarrow \mathbb{R}$ is given by $\rho(r) = \sqrt{r+1}$.
By construction, we have $\operatorname{pr}_{[0, 1]} \circ \alpha^{j} = \operatorname{pr}_{[0, 1]}$ and $(\tilde{g} \circ \alpha^{j})(u, v, t) = e^{j}(\varepsilon^{2}(u+1))$ when $t \in [0, 1]$ is near $0$ or near $1$.
For $j = 0, 1$, we use the technique of integral curves of vector fields on manifolds with boundary (see \cite[Chapter 6 \S 2, pp. 149ff]{hir}) to construct another tubular neighborhood of $V^{j} \times [0, 1]$ in $\tilde{V} \times [0, 1]$, say
\begin{displaymath}
\beta^{j} \colon (-\delta, \delta) \times V^{j} \times [0, 1] \rightarrow \tilde{V} \times [0, 1],
\end{displaymath}
such that $\operatorname{pr}_{[0, 1]} \circ \beta^{j} = \operatorname{pr}_{[0, 1]}$ and $(\tilde{g} \circ \beta^{j})(u, v, t) = e^{j}(\varepsilon^{2}(u+1))$.
By adapting the proof of \cite[Theorem 5.3, p. 112]{hir} we can construct for some open neighborhood $U \subset \tilde{V} \times [0, 1]$ of $\partial V \times [0, 1]$ an isotopy rel $U \cap (\tilde{V} \times \{0, 1\})$ from the inclusion $U \hookrightarrow \tilde{V} \times [0, 1]$ to an embedding $\theta \colon U \rightarrow \tilde{V} \times [0, 1]$ such that $\theta \circ \alpha^{j} = \beta^{j}$ on a neighborhood of $V^{j} \times [0, 1]$ in $\tilde{V} \times [0, 1]$.
A version of the isotopy extension theorem (see \cite[Theorem 1.4, p. 180]{hir}) provides an ambient isotopy rel a neighborhood of $M \times \{0, 1\}$ in $M \times [0, 1]$ from $\operatorname{id}_{M \times [0, 1]}$ to an automorphism $\Theta$ of $M \times [0, 1]$ such that $\Theta \circ \alpha^{j} = \beta^{j}$ for $j = 0, 1$ on a neighborhood of $V^{j} \times [0, 1]$ in $M \times [0, 1]$.
Finally, the desired map $F \colon M \times [0, 1] \rightarrow \mathbb{R}^{2}$ can be defined as
\begin{displaymath}
F(x, t) = \begin{cases}
(G \circ \Theta)(x, t), \text{ if } x \in V, \\
(e^{j}(||(\iota_{0}^{j})^{-1}(x)||^{2}), t), \text{ if } x \in U^{j}.
\end{cases}
\end{displaymath}
\end{proof}

\begin{remark}\label{remark boundary case}
In \cite[Lemma 3.1, p. 291]{sae2}, Cerf's pseudo-isotopy theorem \cite{cer} is used to show that the statement of \Cref{theorem cylinder bordism between constrained morse functions} also holds for $n \geq 6$ and $k > n/2$.
\end{remark}

\subsection{Handle Extension Theorem}\label{handle extension theorem}
Using the techniques of standard Morse functions (see \Cref{theorem existence standard Morse functions}) and forward handles (see \Cref{remark forward handle}) of Gay and Kirby \cite{gaykir}, we prove a handle extension theorem (\Cref{theorem handle extension}) for constrained Morse functions (see also \cite[Chapter 7]{wra}).

Let $W_{0}^{n}$ and $W_{1}^{n}$ be oriented closed manifolds of dimension $n$, and let $W^{n+1}$ be an oriented compact $(n+1)$-manifold with oriented boundary $\partial W = W_{0} \sqcup W_{1}$.
Let us consider the problem of extending a given $k$-constrained Morse function $f_{0} \colon W_{0} \rightarrow \mathbb{R}$ over $W$, i.e. to construct a $k$-constrained Morse function $f_{1} \colon W_{1} \rightarrow \mathbb{R}$, and an oriented $k$-constrained bordism $F \colon W \rightarrow \mathbb{R}^{2}$ from $f_{0}$ to $f_{1}$.
For $k = 1$ this can always be achieved by applying the techniques of \Cref{section cusps} to a suitable generic extension of $f_{0}$ over $W$.
The ``handle extension theorem'' addresses this problem for $k > 1$, provided that $W$ admits a handle decomposition with only handles of a single index contained in $\{k+1, \dots n-k\}$, and that $f_{0}$ is nicely compatible with the attaching maps of the handles of $W$ (see \Cref{standard Morse functions}).
More specifically, the purpose of \Cref{proof of handle extension theorem} is to derive the following result.

\begin{theorem}\label{theorem handle extension}
Let $n \geq 5$ and $k \in \{2, \dots, \lceil n/2 \rceil -1\}$.
Suppose that $\tau \colon W^{n+1} \rightarrow [0, 1]$ is a Morse function with regular value sets $W_{0} = \tau^{-1}(0)$ and $W_{1} = \tau^{-1}(1)$ such that all critical points of $\tau$ have the same Morse index $\lambda \in \{k+1, \dots, n-k\}$ and are contained in the slice $\tau^{-1}(1/2)$.
Furthermore, suppose that $W_{0}^{n}$ is $(k-1)$-connected.
Then, there exists a smooth function $\sigma \colon W \rightarrow \mathbb{R}$ with the following properties:
\begin{enumerate}[(i)]
\item $\sigma$ restricts for every $t \neq 1/2$ to 
a $k$-constrained Morse function
$$\sigma_{t} \colon \tau^{-1}(t) \rightarrow \mathbb{R}.$$
\item $\sigma$ and $\tau$ form the components of an oriented $k$-constrained bordism
\begin{displaymath}
(\sigma, \tau) \colon W \rightarrow \mathbb{R} \times [0, 1]
\end{displaymath}
from $\sigma_{0}$ to $\sigma_{1}$.
\end{enumerate}
\end{theorem}

\begin{remark}\label{remark forward handle}
Following \cite{gaykir}, the main idea behind the construction of the desired function $\sigma$ in \Cref{theorem handle extension} can be illustrated as follows for a bordism $W$ of dimension $n+1=2$ with a single critical point $c$ of index $\lambda = 1$ (see \Cref{fig:main_theorem_forward_handle}).

\begin{figure}[htbp]
  \centering
  \fbox{\begin{tikzpicture}[scale=0.3652]
    \draw (0, 0) node[inner sep=0] {\includegraphics[width=0.7\textwidth]{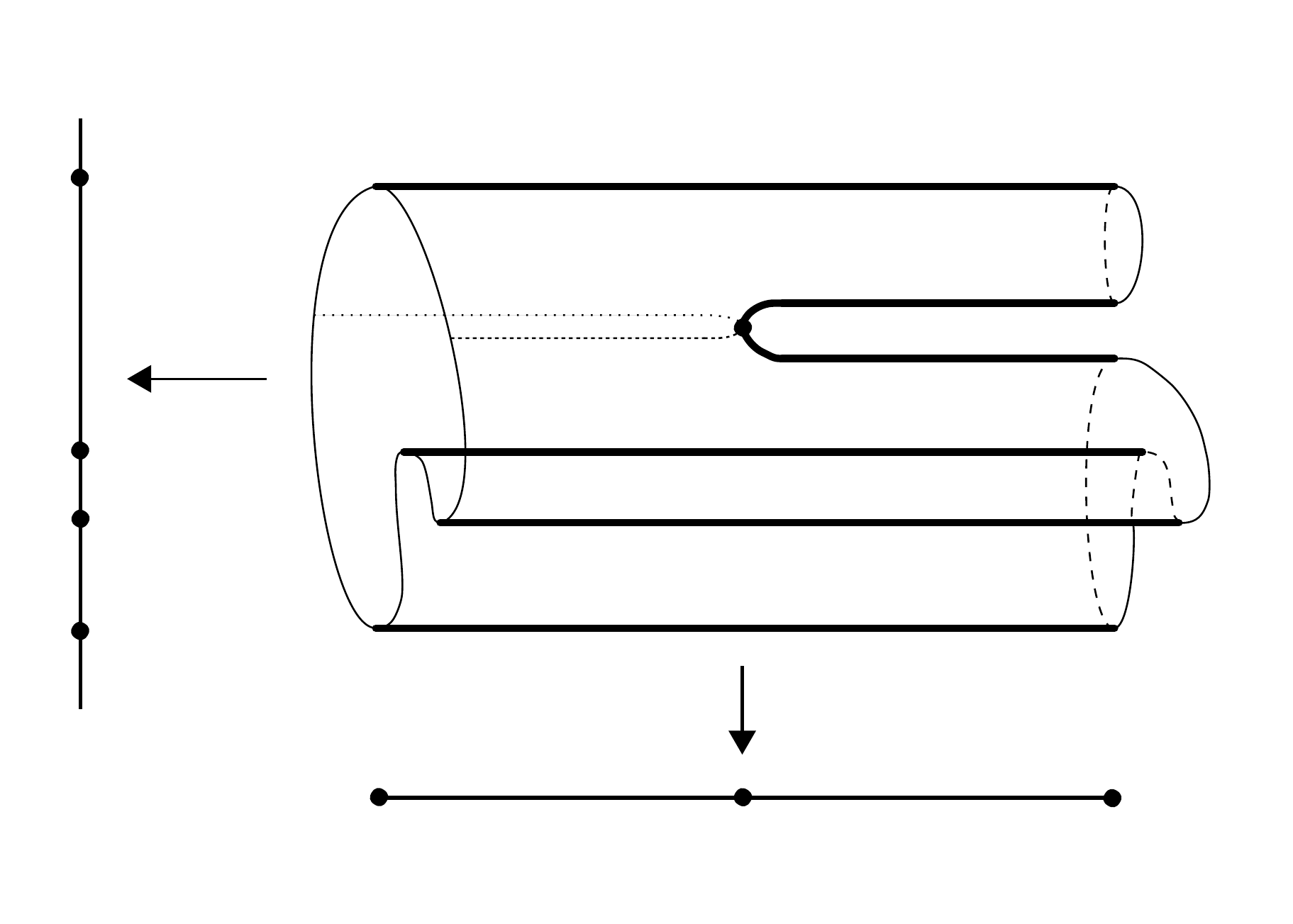}};
    \draw (2.5, 6.4) node {$W^{n+1}$};
    \draw (-5, 6.4) node {$W_{0}$};
    \draw (8.8, 6.4) node {$W_{1}$};
    \draw (1.8, 1.6) node {$c$};
    \draw (-5, -7.2) node {$0$};
    \draw (8.8, -7.2) node {$1$};
    \draw (1.3, -4.5) node {$\tau$};
    \draw (-8.2, 2.2) node {$\sigma$};
    \draw (-11.3, 6.6) node {$\mathbb{R}$};
\end{tikzpicture}}
  \caption{Construction of $\sigma$ as the height function on a $2$-dimensional cobordism $W$. The fold lines of $(\sigma, \tau)$ are marked as bold lines on $W$.
  Note that $\sigma$ restricts to an excellent Morse function on $W_{0}$ that is standard with respect to the left-hand sphere of the critical point $c$ of $\tau$ in $W_{0}$.}
  \label{fig:main_theorem_forward_handle}
\end{figure}
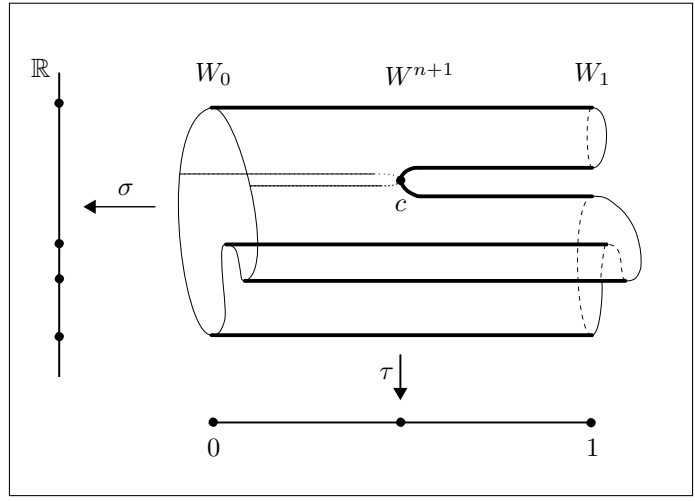

Choose local coordinates $(x, y) \in \mathbb{R}^{\lambda} \times \mathbb{R}^{n+1-\lambda}$ centered at $c$ in which $\tau$ has the normal form $(x, y) \mapsto -||x||^{2} + ||y||^2 + 1/2$, and consider the so-called \emph{forward $\lambda$-handle} (see \cite[Fig. 29]{gaykir})
\begin{displaymath}
\mathbb{R}^{\lambda} \times \mathbb{R}^{n+1-\lambda} \rightarrow \mathbb{R}^{2}, \qquad (x, y) \mapsto (-||x||^{2} + ||y||^2+1/2, y_{1}).
\end{displaymath}
By composition with the diffeomorphism $\mathbb{R}^{2} \rightarrow \mathbb{R}^{2}$, $(a, b) \mapsto (a-b^{2}, b)$, this map can be shown to be a fold map with a single fold line (namely, the $y_{1}$-axis) of absolute index $\operatorname{max}\{\lambda, n-\lambda\}$.
Still working in the local chart around $c$, the proof of \Cref{forward k handle} uses a bump function to modify the forward $\lambda$-handle outside a compact neighborhood of the origin in a way that allows to extend it to the desired function $(\sigma, \tau)$ on all of $W$.
By means of integral curves of a gradient-like vector field of $\tau$ we reduce this extension problem for suitable $t_{-} \in (0, 1/2)$ to the construction of a $k$-constrained Morse function $\sigma_{-} \colon \tau^{-1}(t_{-}) \rightarrow \mathbb{R}$ which is in addition \emph{standard} (see \cite[Section 4]{gaykir}) with respect to the left-hand sphere of the critical point of $\tau$.
Then, $\sigma$ can be taken to be a height function in \Cref{fig:main_theorem_forward_handle}.
Finally, as indicated in \Cref{fig:main_theorem_forward_handle}, the fold lines of $(\sigma, \tau)$ are given by the suspended fold points of $\sigma_{-}$, and one additional indefinite fold line of absolute index $\operatorname{max}\{\lambda, n-\lambda\}$ coming from the forward $\lambda$-handle.
\end{remark}

\begin{remark}
The original intention of \cite{gaykir} is to navigate between so-called \emph{Morse $2$-functions}, i.e., generic maps from a bordism into a surface.
Motivated by the study of Lefschetz fibrations in the context of symplectic geometry, Gay and Kirby focus on Morse $2$-functions without definite fold points, and with connected fibers.
Our achievement in \Cref{theorem handle extension} is to adapt their method to the case that stronger index constraints are imposed on indefinite absolute indices of fold points (compare \cite[Remark 1.6, p. 8]{gaykir}).
\end{remark}

\subsubsection{Standard Morse Functions}\label{standard Morse functions}
In \Cref{theorem existence standard Morse functions} below we recall a result of Gay and Kirby (see \cite[Theorem 4.2]{gaykir}) on the existence of so-called indefinite standard Morse functions.
For this purpose, let $Y_{0}^{n-1}$ and $Y_{1}^{n-1}$ denote nonempty closed manifolds of dimension $n-1 \geq 1$, and let $Y^{n}$ be a connected compact $n$-manifold with boundary $\partial Y = Y_{0} \sqcup Y_{1}$.
For $i = 1, \dots, N$ let
\begin{displaymath}
\phi_{i} \colon L_{i}^{d} \times \operatorname{int} D_{\varepsilon}^{n-d} \rightarrow Y \setminus \partial Y, \qquad \varepsilon > 0,
\end{displaymath}
be pairwise disjoint embeddings, where $L_{i}^{d}$ are closed manifolds of dimension $d < n/2$.
Fix real numbers $0 < z_{1} < \dots < z_{N} < 1$.

\begin{theorem}\label{theorem existence standard Morse functions}
There exists a Morse function $g \colon Y^{n} \rightarrow [0, 1]$ with regular level sets $Y_{0} = f^{-1}(0)$ and $Y_{1} = f^{-1}(1)$, and with the following properties:
\begin{enumerate}[(i)]
\item $g$ is standard with respect to the pairs $(\phi_{i}, z_{i})$, i.e., there exists $\varepsilon' \in (0, \varepsilon)$ such that for $i = 1, \dots, N$,
\begin{displaymath}
g(\phi_{i}(u, v)) = v_{1} + z_{i}, \qquad (u, v) \in L_{i}^{d} \times \operatorname{int} D_{\varepsilon'}^{n-d}.
\end{displaymath}
\item $g$ is indefinite, i.e., $g$ has neither critical points of index $0$ nor of index $n$.
\item Every critical point $c$ of $g$ of index $\leq d$ satisfies $g(c) < z_{1}$, whereas every critical point $c$ of $g$ of index $> d$ satisfies $g(c) > z_{N}$.
\end{enumerate}
\end{theorem}

\begin{remark}
By property $(i)$ of \Cref{theorem existence standard Morse functions}, the submanifold $L_{i} = \phi_{i}(L_{i} \times 0) \subset Y$ lies in the fiber $g^{-1}(z_{i})$ of $g$.
Furthermore, the framing of $L_{i}$ in $Y$ induced by $\phi_{i}$ is nicely compatible with $g$ in such a way that $g$ has no critical points in $\phi_{i}(L_{i} \times \operatorname{int} D_{\varepsilon'}^{n-d})$.
\end{remark}

The following corollary is concerned with the existence of constrained Morse functions that are standard with respect to prescribed pairs $(\phi_{i}, z_{i})$.

\begin{corollary}\label{corollary constrained standard Morse functions}
Let $n \geq 5$ and $k \in \{2, \dots, \lceil n/2 \rceil-1\}$.
Let $M^{n}$ be a $(k-1)$-connected closed $n$-manifold.
Suppose that $L_{1}^{d}, \dots, L_{N}^{d}$ are closed manifolds of dimension $d \in \{k, \dots, \lceil n/2 \rceil-1\}$, and that, for some $\varepsilon > 0$, there exist embeddings
\begin{displaymath}
\phi_{i} \colon L_{i}^{d} \times \operatorname{int} D_{2 \varepsilon}^{n-d} \rightarrow M^{n}, \qquad i \in \{1, \dots, N\},
\end{displaymath}
whose images have pairwise disjoint closures.
Furthermore, let $0 < z_{1} < \dots < z_{N} < 1$ be real numbers.
Then, $M^{n}$ admits a $k$-constrained Morse function $f \colon M^{n} \rightarrow \mathbb{R}$ such that, for some $C > 0$,
\begin{displaymath}
f(\phi_{i}(u, v)) = C \cdot v_{1} + z_{i}, \qquad (u, v) \in L_{i}^{d} \times \operatorname{int} D_{\varepsilon}^{n-d}, \qquad i = 1, \dots, N.
\end{displaymath}
\end{corollary}

\begin{proof}
Choose disjoint embeddings $\iota_{0}, \iota_{1} \colon D^{n} \rightarrow M^{n}$ which are also disjoint to the closures of the images of the embeddings $\phi_{i}$.
We apply \Cref{theorem existence standard Morse functions} to the compact $n$-manifold $Y^{n} = M^{n} \setminus \iota_{0}(\operatorname{int} D^{n}) \cup \iota_{1}(\operatorname{int} D^{n})$ with boundary components $Y_{0}^{n-1} = \iota_{0}(S^{n-1})$ and $Y_{1}^{n-1} = \iota_{1}(S^{n-1})$ to obtain a Morse function $g \colon Y^{n} \rightarrow [0, 1]$ with regular level sets $Y_{0} = g^{-1}(0)$ and $Y_{1} = g^{-1}(1)$ satisfying properties $(i)$ to $(iii)$.
By construction, $Y^{n}$, $Y_{0}^{n-1}$ and $Y_{1}^{n-1}$ are simply connected, and we have
\begin{displaymath}
\qquad H_{j}(Y, Y_{0}; \mathbb{Z}) = H_{j}(Y, Y_{1}; \mathbb{Z}) = 0, \qquad j = 1, \dots, k-1.
\end{displaymath}
Hence, in view of properties $(ii)$ and $(iii)$, we can use standard arguments of Morse theory (including \cite[Theorem 8.1, page 100]{mil2} as well as the techniques used in the proof of \cite[Theorem 7.8, p. 97]{mil2}) to eliminate all critical points of $g$ of indices $1, \dots, k-1$ by a homotopy with compact support in $g^{-1}((0, z_{1}))$, and all critical points of $g$ of indices $n-k+1, \dots, n-1$ by a homotopy with compact support in $g^{-1}((z_{N}, 1))$.
Therefore, we may replace property $(ii)$ by the stronger assumption that all indices of critical points of $g$ are contained in the set $\{k, \dots, n-k\}$.

Set $C = \varepsilon'/\varepsilon \in (0, 1)$.
By means of the isotopy extension lemma \cite[Theorem 1.3, p. 180]{hir} we construct a diffeomorphism $\rho \colon \operatorname{int} D^{n-d}_{2\varepsilon} \stackrel{\cong}{\longrightarrow} \operatorname{int} D^{n-d}_{2\varepsilon}$ such that $\rho(v) = C \cdot v$ for $||v|| < \varepsilon$ and $\rho(v) = v$ for $||v|| > 3 \varepsilon/2$.
Let $\Xi$ denote the automorphism of $Y^{n}$ which extends the automorphisms $\phi_{i} \circ (\operatorname{id}_{L_{i}} \times \rho) \circ \phi_{i}^{-1}$ of $\phi_{i}(L_{i}^{d} \times \operatorname{int} D_{2 \varepsilon}^{n-d})$, $i = 1, \dots, N$, by the identity map.
Then, by property $(i)$ of $g$,
\begin{displaymath}
(g \circ \Xi)(\phi_{i}(u, v)) = C \cdot v_{1} + z_{i}, \qquad (u, v) \in L_{i}^{d} \times \operatorname{int} D_{\varepsilon}^{n-d}, \qquad i = 1, \dots, N.
\end{displaymath}
Finally, we use \cite[Lemma 3.7, p. 26]{mil2} to extend $g \circ \Xi \colon Y^{n} \rightarrow [0, 1]$ to the desired Morse function $f \colon M^{n} \rightarrow \mathbb{R}$ with two additional critical points, namely one at $\iota_{0}(0)$ of index $0$, and one at $\iota_{1}(0)$ of index $n$.
\end{proof}

\begin{remark}
\Cref{corollary constrained standard Morse functions} will eventually be brought to bear in the proof of \Cref{theorem handle extension} (see \Cref{proof of handle extension theorem}) in the case where $L_{i} = S^{d}$, and the embeddings $\phi_{i}$ play the role of attaching maps of $(d+1)$-handles.
\end{remark}

\subsubsection{Constructing Fold Maps from Local Handles into the Plane}\label{constructing fold maps from local handles into the plane}
We refer to \cite[Section 7.2, pp. 178ff]{wra} for complete details of the construction presented here.

Fix a pair $(m, \lambda)$ consisting of integers $m \geq 2$ and $\lambda \in \{1, \dots, m-1\}$.
Define a Morse function $\mu \colon \mathbb{R}^{m} \rightarrow \mathbb{R}$ with a single critical point of Morse index $\lambda$ at the origin of $\mathbb{R}^{m} = \mathbb{R}^{\lambda} \times \mathbb{R}^{m-\lambda}$ by
\begin{displaymath}
\mu \colon \mathbb{R}^{\lambda} \times \mathbb{R}^{m-\lambda} \rightarrow \mathbb{R}, \qquad \mu(p, q) = -||p||^{2} + ||q||^{2}.
\end{displaymath}
A gradient-like vector field $\upsilon$ for $\mu$ is given by
\begin{displaymath}
\upsilon \colon \mathbb{R}^{\lambda} \times \mathbb{R}^{m-\lambda} \rightarrow \mathbb{R}^{m}, \qquad \upsilon(p, q) = (-p, q).
\end{displaymath}
Note that the flow $\eta$ of $\upsilon$ is given by
\begin{displaymath}
\eta \colon \mathbb{R}^{\lambda} \times \mathbb{R}^{m-\lambda} \times \mathbb{R} \rightarrow \mathbb{R}^{m}, \qquad \eta(p, q, t) = (e^{-t} p, e^{t} q).
\end{displaymath}
Indeed, for any point $x = (p, q) \in \mathbb{R}^{\lambda} \times \mathbb{R}^{m-\lambda}$, the integral curve
\begin{displaymath}
\eta_{x} \colon \mathbb{R} \rightarrow \mathbb{R}^{\lambda} \times \mathbb{R}^{m-\lambda}, \qquad \eta_{x}(t) = \eta(p, q, t) = (e^{-t} p, e^{t} q),
\end{displaymath}
satisfies $\eta_{x}(0) = x$ and $\eta_{x}'(t) = (-e^{-t} p, e^{t} q) = \upsilon(\eta_{x}(t))$ for all $t \in \mathbb{R}$.

\begin{definition}\label{k-handle definition}
Given $\varepsilon, \delta > 0$, we define the \emph{local $(\varepsilon, \delta)$-handle} by
\begin{displaymath}
H^{\varepsilon}_{\delta} = \{(p, q) \in \mathbb{R}^{\lambda} \times \mathbb{R}^{m-\lambda}; \; -\delta^{2} \leq -||p||^{2} + ||q||^{2} \leq \delta^{2}, \; ||p|| \cdot ||q|| < \varepsilon \cdot \sqrt{\varepsilon^{2} + \delta^{2}}\}.
\end{displaymath}
\end{definition}

The following two lemmas (see \cite[Proposition 7.2.4, p. 180]{wra}) are essentially observed in the proof of \cite[Theorem 3.12, page 30]{mil2}.
Note that one part of the boundary of $H^{\varepsilon}_{\delta}$ is contained in the slice $\mu^{-1}(-\delta^{2})$, and forms a tubular neighborhood of the left-hand sphere of the unique critical point of $\mu$.

\begin{lemma}\label{handle properties}
The local $(\varepsilon, \delta)$-handle $H^{\varepsilon}_{\delta}$ is an $m$-dimensional submanifold of $\mathbb{R}^{m} = \mathbb{R}^{\lambda} \times \mathbb{R}^{m-\lambda}$ with boundary $\partial H^{\varepsilon}_{\delta} = L^{\varepsilon}_{\delta} \sqcup R^{\varepsilon}_{\delta}$, where $L^{\varepsilon}_{\delta} := H^{\varepsilon}_{\delta} \cap \mu^{-1}(- \delta^{2})$ and $R^{\varepsilon}_{\delta} := H^{\varepsilon}_{\delta} \cap \mu^{-1}(+ \delta^{2})$.
There are diffeomorphisms
\begin{align*}
\lambda^{\varepsilon}_{\delta} \colon S^{\lambda-1} \times \operatorname{int} D_{\varepsilon}^{m-\lambda}
&\stackrel{\cong}{\longrightarrow} L^{\varepsilon}_{\delta}, \qquad \lambda^{\varepsilon}_{\delta}(u, v) = (\sqrt{||v||^{2} + \delta^{2}} \cdot u, v), \\
\rho^{\varepsilon}_{\delta} \colon \operatorname{int} D_{\varepsilon}^{\lambda} \times S^{m-\lambda-1} &\stackrel{\cong}{\longrightarrow} R^{\varepsilon}_{\delta}, \qquad \rho^{\varepsilon}_{\delta}(u, v) = (u, \sqrt{||u||^{2} + \delta^{2}} \cdot v).
\end{align*}
\end{lemma}

Let $Z:= (\mathbb{R}^{\lambda} \times \{0\}) \cup (\{0\} \times \mathbb{R}^{m-\lambda})$.
The flow $\eta$ of $\mu$ is used in the following lemma to express $H^{\varepsilon}_{\delta} \setminus Z$ in terms of $L^{\varepsilon}_{\delta} \setminus Z$.

\begin{lemma}\label{lemma handle diffeomorphism}
Let $\varepsilon, \delta > 0$.
We define manifolds with boundary by
\begin{align*}
X^{\varepsilon}_{\delta} &:= \{(v, t) \in \mathbb{R}^{m-\lambda} \times \mathbb{R}; \; 0 < ||v|| < \varepsilon, \; 0 \leq t \leq 1/2 \cdot \operatorname{log}(1+ \delta^{2}/||v||^{2})\}, \\
Y^{\varepsilon}_{\delta} &:= \{(x, t) = (p, q, t) \in L^{\varepsilon}_{\delta} \times \mathbb{R}; \; q \neq 0, \; 0 \leq t \leq \operatorname{log}(||p||/||q||)\}.
\end{align*}
The diffeomorphism $\lambda^{\varepsilon}_{\delta}$ (see \Cref{handle properties}) induces a diffeomorphism
\begin{displaymath}
S^{\lambda-1} \times X^{\varepsilon}_{\delta} \stackrel{\cong}{\longrightarrow} Y^{\varepsilon}_{\delta}, \qquad (u, v, t) \mapsto (\lambda^{\varepsilon}_{\delta}(u, v), t).
\end{displaymath}
The flow $\eta$ of $\mu$ restricts to a diffeomorphism
\begin{displaymath}
Y^{\varepsilon}_{\delta} \stackrel{\cong}{\longrightarrow} H^{\varepsilon}_{\delta} \setminus Z, \qquad (x, t) = (p, q, t) \mapsto \eta(p, q, t) = (e^{-t} p, e^{t} q).
\end{displaymath}
By composition we obtain a diffeomorphism
\begin{displaymath}
\Lambda^{\varepsilon}_{\delta} \colon S^{\lambda-1} \times X^{\varepsilon}_{\delta} \stackrel{\cong}{\longrightarrow} H^{\varepsilon}_{\delta} \setminus Z, \qquad (u, v, t) \mapsto (e^{-t} \cdot \sqrt{||v||^{2} + \delta^{2}} \cdot u, e^{t} \cdot v).
\end{displaymath}
\end{lemma}

The next result implements the technique of forward handles (see \cite[Fig. 29]{gaykir}).

\begin{proposition}\label{forward k handle}
Let $\varepsilon > 0$.
If $\delta > 0$ is sufficiently small, then there exists a smooth function $\nu \colon H^{\varepsilon}_{\delta} \rightarrow \mathbb{R}$ with the following properties:
\begin{enumerate}[(i)]
\item The map $(\mu|_{H^{\varepsilon}_{\delta}}, \nu) \colon H^{\varepsilon}_{\delta} \rightarrow \mathbb{R}^{2}$ is a fold map whose singular locus is the fold line $S_{\delta} := \{0\} \times [-\delta, \delta] \times \{0\} \subset \mathbb{R}^{\lambda} \times \mathbb{R} \times \mathbb{R}^{m-1-\lambda}$ of absolute index $\operatorname{max}\{\lambda, m-1-\lambda\}$.
\item For $0 < |s| \leq \delta^{2}$ the set of critical points of the restriction $\nu|_{H^{\varepsilon}_{\delta, s}} \colon H^{\varepsilon}_{\delta, s} \rightarrow \mathbb{R}$ of $\nu$ to the slice $H^{\varepsilon}_{\delta, s} := H^{\varepsilon}_{\delta} \cap \mu^{-1}(s)$ is given by
\begin{displaymath}
H^{\varepsilon}_{\delta, s} \cap S_{\delta} =
\begin{cases}
\emptyset, \quad s < 0, \\
\{(0, \pm\sqrt{s}, 0)\} =: \{x^{\pm}_{s}\}, \quad s > 0.
\end{cases}
\end{displaymath}
We have $H^{\varepsilon}_{\delta, -\delta^{2}} = L^{\varepsilon}_{\delta}$, and the restriction $\nu|_{L^{\varepsilon}_{\delta}} \colon L^{\varepsilon}_{\delta} \rightarrow \mathbb{R}$ is the projection $(p,q) \mapsto q_{1}$.
Moreover, if $s > 0$, then $x^{-}_{s}$ is a non-degenerate critical point of $\nu|_{H^{\varepsilon}_{\delta, s}}$ of Morse index $\lambda$ at level $-\sqrt{s}$, and $x^{-}_{s}$ is a non-degenerate critical point of $\nu|_{H^{\varepsilon}_{\delta, s}}$ of Morse index $m-\lambda-1$ at level $+\sqrt{s}$.
\item There exists $\varepsilon' \in (0, \varepsilon)$ such that $\nu(\eta_{x}(t)) = \nu(\eta_{x}(0))$ for all $(x, t) \in Y^{\varepsilon}_{\delta} \setminus \overline{Y^{\varepsilon'}_{\delta}}$.
\end{enumerate}
\end{proposition}

\begin{proof}
Let $\varepsilon_{0} := \varepsilon/3$.
First, we construct a smooth map $\nu_{<} \colon H^{\varepsilon_{0}}_{\delta} \rightarrow \mathbb{R}$ on the open subset $H^{\varepsilon_{0}}_{\delta} \subset H^{\varepsilon}_{\delta}$, and a smooth map $\nu_{>} \colon H^{\varepsilon}_{\delta}\setminus Z \rightarrow \mathbb{R}$ on the open subset $H^{\varepsilon}_{\delta}\setminus Z \subset H^{\varepsilon}_{\delta}$ such that $\nu_{<}$ and $\nu_{>}$ agree on the intersection $H^{\varepsilon_{0}}_{\delta} \cap (H^{\varepsilon}_{\delta}\setminus Z) = H^{\varepsilon_{0}}_{\delta} \setminus Z$.
Let $\nu_{<}$ be the smooth map given by the projection to the $(\lambda+1)$-th coordinate:
\begin{displaymath}
\nu_{<} \colon H^{\varepsilon_{0}}_{\delta} \rightarrow \mathbb{R}, \qquad x = (p, q) \mapsto q_{1}.
\end{displaymath}
For the construction of $\nu_{>}$, choose a smooth map $\xi \colon [0, \infty) \rightarrow \mathbb{R}$ such that
$\xi([0, \infty)) \subset [0, 1]$, $\xi(r) = 1$ for $r < \varepsilon_{0}^{2}$ and $\xi(r) = 0$ for $r > (2 \varepsilon_{0})^{2}$.
Then, use \Cref{lemma handle diffeomorphism} to define the smooth map $\nu_{>} \colon H^{\varepsilon}_{\delta}\setminus Z \rightarrow \mathbb{R}$ by requiring that
$$(\nu_{>} \circ \eta)(x, t) = e^{t \cdot \xi(||q||^{2})} q_{1}$$
for all $(x, t) = (p, q, t) \in Y^{\varepsilon}_{\delta}$.
Since $\nu_{>}$ and $\nu_{<}$ agree at all $H^{\varepsilon_{0}}_{\delta} \setminus Z$
$$(\nu_{>} \circ \eta)(x, t) = e^{t \cdot \xi(||q||^{2})} q_{1} = e^{t} q_{1} = (\nu_{<} \circ \eta)(x, t)$$
for all $(x, t) = (p, q, t) \in Y^{\varepsilon}_{\delta}$ because $||q|| < \varepsilon_{0}$ whenever $x = (p, q) \in L^{\varepsilon_{0}}_{\delta}$.

It remains to show that the glued map
\begin{displaymath}
\nu \colon H^{\varepsilon}_{\delta} \rightarrow \mathbb{R}, \qquad \nu(x) = \begin{cases}
\nu_{<}(x), \quad \text{if} \quad x \in H^{\varepsilon_{0}}_{\delta}, \\
\nu_{>}(x), \quad \text{if} \quad x \in H^{\varepsilon}_{\delta} \setminus Z,
\end{cases}
\end{displaymath}
satisfies the desired properties when $\delta > 0$ is sufficiently small.

$(i)$
By construction, $(\mu|_{H^{\varepsilon}_{\delta}}, \nu)$ restricts on $H^{\varepsilon_{0}}_{\delta}$ to the forward $\lambda$-handle discussed in \Cref{remark forward handle}.
Let us check that $(\mu|_{H^{\varepsilon}_{\delta}}, \nu)$ is a submersion on $H^{\varepsilon}_{\delta} \setminus Z$.
Since precomposition with the diffeomorphisms $\Phi^{\varepsilon}_{-}$ of \Cref{handle properties} is constant in the variable $u \in S^{\lambda-1}$, we have to show that the following map is a submersion:
\begin{displaymath}
F \colon X^{\varepsilon}_{\delta} \rightarrow \mathbb{R}^{2}, \qquad (v, t) \mapsto (-\delta^{2}e^{-2t} + ||v||^{2}(e^{2t} - e^{-2t}), e^{t \cdot \xi(||v||^{2})} v_{1}).
\end{displaymath}
Writing $r := ||v||^{2}$, the Jacobian of $F$ at $(v, t)$ is given by the $2 \times (m-\lambda+1)$-matrix
{\footnotesize
\begin{displaymath}
\left(\begin{matrix}
2 v_{1}(e^{2t} - e^{-2t}) & 2 v_{2}(e^{2t} - e^{-2t}) & \dots & 2 v_{m-\lambda}(e^{2t} - e^{-2t}) & 2 \delta^{2}e^{-2t} + 2r(e^{2t} + e^{-2t}) \\
(1 + 2t\xi'(r)v_{1}^{2})e^{t \cdot \xi(r)} & 2t\xi'(r)v_{1}v_{2}e^{t \cdot \xi(r)} & \dots & 2t\xi'(r)v_{1}v_{m-\lambda}e^{t \cdot \xi(r)} & e^{t \cdot \xi(r)} v_{1} \xi(r)
\end{matrix}
\right)
\end{displaymath}}
For $i \in \{2, \dots, m-\lambda\}$ the determinant of the $2 \times 2$-submatrix given by the first and the $i$-th column is
\begin{displaymath}
\operatorname{det}\left(
\begin{matrix}
2 v_{1}(e^{2t} - e^{-2t}) & 2 v_{i}(e^{2t} - e^{-2t}) \\
(1 + 2t\xi'(r)v_{1}^{2})e^{t \cdot \xi(r)} & 2t\xi'(r)v_{1}v_{i}e^{t \cdot \xi(r)}
\end{matrix}\right)
= -2 v_{i}(e^{2t} - e^{-2t})e^{t \cdot \xi(r)}.
\end{displaymath}
This determinant vanishes if and only if $t = 0$ or $v_{i} = 0$. Thus, the rank of the Jacobian of $F$ at $(v, t)$ remains to be investigated only in the case that $t = 0$ or $v_{2} = \dots = v_{m-\lambda} = 0$.
For this purpose, we consider the $2 \times 2$-submatrix of the Jacobian of $F$ at $(v, t)$ consisting of the first and the last column, which is
\begin{align*}
&\operatorname{det}\left(
\begin{matrix}
2 v_{1}(e^{2t} - e^{-2t}) & 2 \delta^{2}e^{-2t} + 2r(e^{2t} + e^{-2t}) \\
(1 + 2t\xi'(r)v_{1}^{2})e^{t \cdot \xi(r)} & e^{t \cdot \xi(r)} v_{1} \xi(r)
\end{matrix}\right) \\
&= 2 v_{1}^{2}(e^{2t} - e^{-2t})e^{t \cdot \xi(r)} \xi(r) - (2 \delta^{2}e^{-2t} + 2r(e^{2t} + e^{-2t}))(1 + 2t\xi'(r)v_{1}^{2})e^{t \cdot \xi(r)}.
\end{align*}
If $t = 0$, then this determinant is further equal to $-(2 \delta^{2} + 4r)$, which is always negative.
In the following, we assume that $v_{2} = \dots = v_{m-\lambda} = 0$.
Then, $r = ||v||^{2} = v_{1}^{2}$, and
the determinant vanishes if and only if the following term vanishes
\begin{displaymath}
(\ast) \qquad r(e^{2t} - e^{-2t}) \xi(r) - (\delta^{2}e^{-2t} + r(e^{2t} + e^{-2t}))(1 + 2t\xi'(r)r).
\end{displaymath}
If $r < \varepsilon_{0}^{2}$, then $\xi(r) = 1$ and $\xi'(r) = 0$, and $(\ast)$ reduces to
\begin{displaymath}
r(e^{2t} - e^{-2t}) - (\delta^{2}e^{-2t} + r(e^{2t} + e^{-2t})) = -(2r + \delta^{2})e^{-2t} < 0.
\end{displaymath}
In the following, we will assume that $\varepsilon_{0}^{2} \leq r$ ($< \varepsilon^{2}$).
Choose $\delta > 0$ so small that
\begin{displaymath}
\operatorname{log}(1 + \delta^{2}/\varepsilon_{0}^{2}) < \operatorname{min} \{1/(2 \varepsilon^{2} \operatorname{max}_{s \in \mathbb{R}}|\xi'(s)|), \operatorname{sinh}^{-1}(1/4)\}.
\end{displaymath}
for all $(v, t) = (v_{1}, 0, \dots, 0, t) \in X^{\varepsilon}_{\delta}$ with $\varepsilon_{0}^{2} \leq r = v_{1}^{2}$ that
\begin{displaymath}
0 \leq t < \operatorname{min} \{1/(4 \varepsilon^{2} \operatorname{max}_{s \in \mathbb{R}}|\xi'(s)|), 1/2 \cdot \operatorname{sinh}^{-1}(1/4)\}.
\end{displaymath}
Consequently, $1 + 2t\xi'(r)r > 1/2 > e^{2t} - e^{-2t} \geq 0$.
Therefore, we obtain the following estimate for the expression $(\ast)$:
\begin{displaymath}
(\ast) \leq r(e^{2t} - e^{-2t}) - r(1 + 2t\xi'(r)r) = r\left[(e^{2t} - e^{-2t}) - (1 + 2t\xi'(r)r)\right] < 0.
\end{displaymath}

$(ii)$
Let $0 < |s| \leq \delta^{2}$.
Any critical point of $\nu|_{H^{\varepsilon}_{\delta, s}}$ is necessarily a critical point of the map $(\mu|_{H^{\varepsilon}_{\delta}}, \nu)$, whose singular locus $S_{\delta}$ has been determined in part $(i)$.
If $s < 0$, then $H^{\varepsilon}_{\delta, s} \cap S_{\delta} = \emptyset$, and the claim about $\nu|_{L^{\varepsilon}_{\delta}}$ holds by construction.
If $s > 0$, then the points of $H^{\varepsilon}_{\delta, s} \cap S_{\delta} = \{(0, \pm\sqrt{s}, 0)\} =: \{x^{\pm}_{s}\}$ can be checked to be non-degenerate critical points of $\nu|_{H^{\varepsilon}_{\delta, s}}$ with the desired properties by a straightforward computation.
For this purpose, note that $x^{\pm}_{s}$ is contained in the open subset $R^{\varepsilon_{0}}_{\sqrt{s}} \subset H^{\varepsilon}_{\delta, s}$ (see \Cref{handle properties}).
On $R^{\varepsilon_{0}}_{\sqrt{s}}$, $\nu$ agrees with $\nu_{<}$, and can be studied near $x^{\pm}_{s}$ by means of 
$$\rho^{\varepsilon_{0}}_{\sqrt{s}} \colon \operatorname{int} D_{\varepsilon_{0}}^{\lambda} \times S^{m-\lambda-1} \stackrel{\cong}{\longrightarrow} R^{\varepsilon_{0}}_{\sqrt{s}}, \qquad \rho^{\varepsilon_{0}}_{\sqrt{s}}(u, v) = (u, \sqrt{||u||^{2} + s} \cdot v),$$
and the inverse of the stereographic projection,
\begin{align*}
\sigma_{\pm} \colon \mathbb{R}^{m-\lambda-1} &\stackrel{\cong}{\longrightarrow} S^{m-\lambda-1}\setminus \{(\mp 1, 0, \dots, 0)\}, \\
w = (w_{1}, \dots, w_{m-\lambda-1}) &\mapsto \left(\mp\frac{|w|^{2}-1}{|w|^{2}+1}, \frac{2w_{1}}{|w|^{2}+1}, \dots, \frac{2w_{m-\lambda-1}}{|w|^{2}+1}\right).
\end{align*}

$(iii)$
Let $\varepsilon' := 2 \varepsilon_{0}$ and $(x, t) = (p, q, t) \in Y^{\varepsilon}_{\delta} \setminus \overline{Y^{\varepsilon'}_{\delta}}$.
Then, it follows from $\eta_{x}(t) \in H^{\varepsilon}_{\delta} \setminus Z$ and $||q||^{2} > (2 \varepsilon_{0})^{2}$ that
\begin{displaymath}
\nu(\eta_{x}(t)) = (\nu_{>} \circ \eta)(x, t) = e^{t \cdot \xi(||q||^{2})} q_{1} = q_{1} = \nu(\eta_{x}(0)).
\end{displaymath}
\end{proof}

\subsubsection{Proof of \Cref{theorem handle extension}}\label{proof of handle extension theorem}

Without loss of generality we may assume that $k+1 \leq \lambda \leq \lceil n/2 \rceil$.
(In fact, it suffices to show that $W_{1}$ is $(k-1)$-connected.
By an argument similar to \cite[Remark 1, p. 70]{mil2} we can show that $W_{1}$ is simply connected.
Then the claim follows because
the effect of a $p$-surgery on a closed manifold of dimension $d \geq p+2$ does not affect the homology groups in dimensions strictly below $\operatorname{min}\{p, d-p-1\}$. In our case, $p = \lambda -1 \in \{k, \dots, n-k-1\}$.)

Let $\xi$ be a gradient-like vector fields of $\tau$ (see \cite[Definition 3.2, p. 20]{mil2}).
Let $c_{1}, \dots, c_{N}$ be the critical points of $\tau$.
For every $i = 1, \dots, N$ there exist open neighborhoods $U_{i}$ of $c_{i} \in W$ and $V_{i}$ of $0 \in \mathbb{R}^{n+1}$, and a chart $\psi_{i} \colon U_{i} \stackrel{\cong}{\longrightarrow} V_{i}$ of $W$ with $\psi_{i}(c_{i}) = 0$, such that the pair $(\tau, \xi)$ corresponds via $\psi_{i}$ to the pair $(\mu, \upsilon)$ introduced at the beginning of \Cref{constructing fold maps from local handles into the plane}), i.e.,
\begin{align*}
\tau \circ \psi_{i}^{-1} = \mu|_{V_{i}} + 1/2, \\
d\psi_{i} \circ \xi \circ \psi_{i}^{-1} = \upsilon|_{V_{i}}.
\end{align*}

We choose $\varepsilon, \delta > 0$ so small that the local $(2\varepsilon, \delta)$-handle $H^{2\varepsilon}_{\delta}$ of \Cref{k-handle definition} is contained in $V_{i}$ for all $i = 1, \dots, N$.
By choosing $\delta > 0$ small enough, we may in addition assume that there exists a smooth map $\nu \colon H_{\delta}^{\varepsilon} \rightarrow \mathbb{R}$ with the properties $(i)$ to $(iii)$ of \Cref{forward k handle}.

We apply \Cref{corollary constrained standard Morse functions} to the $(k-1)$-connected closed $n$-manifold $M^{n} = \tau^{-1}(1/2 - \delta^{2})$, the left-hand spheres $L_{i} = S^{\lambda-1}$ of dimension $d = \lambda-1$, some fixed real numbers $0 < z_{1} < \dots < z_{N} < 1$, and the embeddings
\begin{displaymath}
\phi_{i} \colon S^{\lambda-1} \times \operatorname{int} D_{2 \varepsilon}^{n-\lambda+1} \stackrel{\lambda^{2 \varepsilon}_{\delta}}{\longrightarrow} L^{2 \varepsilon}_{\delta} \stackrel{\psi_{i}^{-1}}{\longrightarrow} U_{i} \cap M^{n} \hookrightarrow M^{n}, \qquad i \in \{1, \dots, N\},
\end{displaymath}
to obtain a $k$-constrained Morse function $f \colon M^{n} \rightarrow \mathbb{R}$ such that, for some $C > 0$,
\begin{displaymath}
f(\phi_{i}(u, v)) = C \cdot v_{1} + z_{i}, \qquad (u, v) \in S^{\lambda-1} \times \operatorname{int} D_{\varepsilon}^{n-\lambda+1}, \qquad i = 1, \dots, N.
\end{displaymath}

Note that it suffices to construct the desired map $\sigma$ in a neighborhood of $\tau^{-1}(1/2)$ in $W$.
In fact, a smooth function $\sigma \colon \tau^{-1}([1/2 - \delta^{2}, 1/2 + \delta^{2}]) \rightarrow \mathbb{R}$ with the desired properties is given by
\begin{displaymath}
\sigma(w) = \begin{cases}
C \cdot \nu(\psi_{i}(w)) + z_{i}, \qquad \text{if} \; w \in \psi_{i}^{-1}(H_{\delta}^{\varepsilon}) \text{ for some } i = 1, \dots, N, \\
f(M \cap \Gamma_{w}), \qquad \text{else},
\end{cases}
\end{displaymath}
where $\Gamma_{w} \subset W$ denotes the image of the uniquely determined maximally extended integral curve through $w$ with respect to $\xi$.
To show that $\sigma$ is a well-defined smooth function, use the behavior of $f$ on $\psi_{i}^{-1}(L_{\delta}^{\varepsilon})$ to that of $\nu$ on $L_{\delta}^{\varepsilon}$, and use property $(iii)$ of $\nu$ in \Cref{forward k handle}.
The desired properties of $\sigma$ follow from the properties $(i)$ and $(ii)$ of $\nu$ in \Cref{forward k handle}.

This completes the proof of \Cref{theorem handle extension}.

\subsection{Stein Factorization}\label{stein factorization}
The importance of Stein factorization for the global study of singularities of smooth maps was first realized when Burlet and de Rham \cite{burder} used it as a tool to study special generic maps of $3$-manifolds into the plane.

We recall the concept of Stein factorization of an arbitrary continuous map $f \colon X \rightarrow Y$ between topological spaces.
Define an equivalence relation $\sim_{f}$ on $X$ as follows.
Two points $x_{1}, x_{2} \in X$ are called equivalent, $x_{1} \sim_{f} x_{2}$, if they are mapped by $f$ to the same point $y := f(x_{1}) = f(x_{2}) \in Y$, and lie in the same connected component of $f^{-1}(y)$.
The quotient map $\pi_{f} \colon X \rightarrow X/\sim_{f}$ gives rise to a unique set-theoretic factorization of $f$ of the form

\begin{center}
\begin{tikzpicture}
  \matrix (m) [matrix of math nodes,row sep=3em,column sep=4em,minimum width=2em]
  {
     X & Y \\
     X/\sim_{f} &  \\};
  \path[-stealth]
    (m-1-1) edge node [left] {$\pi_{f}$} (m-2-1)
            edge node [above] {$f$} (m-1-2)
    (m-2-1) edge node [below] {$\overline{f}$} (m-1-2);
\end{tikzpicture}.
\end{center}

If we equip the quotient space $X_{f} := X/\sim_{f}$ with the quotient topology induced by the surjective map $\pi_{f} \colon X \rightarrow X_{f}$,
then it follows that the maps $\pi_{f}$ and $\overline{f}$ are continuous.

In the following, both the above diagram and the quotient space $X_{f} = X/\sim_{f}$ will be referred to as the \emph{Stein factorization} of $f$.

In generalization of \cite[Lemma 2.1, p. 290]{sae2}, the following result clarifies the structure of Stein factorization for $k$-constrained generic bordisms $W^{m} \rightarrow \mathbb{R}^{2}$ for $m \geq 3$ and $k > 1$.

\begin{theorem}\label{stein factorization for fold maps with indefinite fold lines}
Let $(W^{m}, M_{1}, M_{2})$ be a smooth manifold triad of dimension $m := \operatorname{dim} W \geq 3$.
Suppose that $F \colon W^{m} \rightarrow \mathbb{R}^{2}$ is an oriented $k$-constrained generic bordism, and that $F$ is a stable map.

If $k > 1$, then the Stein factorization $W_{F} = W/\sim_{F}$ of $F$ can be given the structure of a compact smooth manifold of dimension $2$ with corners in such a way that $\pi_{F} \colon W \rightarrow W_{F}$ is a
generic smooth map and $\overline{F} \colon W_{F} \rightarrow \mathbb{R}^{2}$ is an immersion.
Furthermore, if $D(F)$ denotes the union of the definite fold lines of $F$, then the boundary of $W_{F}$ decomposes as
\begin{displaymath}
\partial W_{F} = \pi_{F}(\partial W) \cup \pi_{F}(D(F)),
\end{displaymath}
where $\pi_{F}(\partial W) \cap \pi_{F}(D(F)) = \pi_{F}(\partial W \cap D(F))$ is the set of corners of $W_{F}$, and $\pi_{F}$ restricts to an embedding $D(F) \rightarrow \partial W_{F}$.
\end{theorem}

\begin{proof}
The claims follow from \cite[Theorem 2.2, p. 2609]{kobsae}.
Note that the only local neighborhoods of points in $W_{F}$ that can occur are those of types $(a)$, $(b1)$, $(b3)$, $(c2)$ and $(d2)$ in \cite[Figure 1, p. 2610]{kobsae} because we have excluded fold points of absolute index $m-2$.
This implies that $W_{F}$ is a topological $2$-manifold.
Finally, the desired smooth structure on $W_{F}$ is induced by requiring $\overline{F}$ to be a local diffeomorphism.
\end{proof}

\section{Proof of \Cref{theorem constrained generic bordism and connective bordism}}\label{proof of main theorem}

Fix integers $n \geq 3$ and $1 \leq k < n$.
\par\medskip

\Cref{lemma epsilon homomorphism} below is a straightforward generalization of \cite[Lemma 3.2, p. 292]{sae2}, and will be used frequently.
Note that \Cref{lemma epsilon homomorphism} can be considered as a version of \Cref{theorem handle extension} for the case that all critical points of $\tau$ have the same Morse index $\lambda \in \{1, m-1\}$.

\begin{lemma}\label{lemma epsilon homomorphism}
Let $f \colon M^{n} \rightarrow \mathbb{R}$ be a $k$-constrained Morse function.
Then, there exist a $k$-constrained Morse function $g \colon \sharp(M^{n}) \rightarrow \mathbb{R}$ and an oriented $k$-constrained bordism $G \colon W^{n+1} \rightarrow \mathbb{R}^{2}$ from $g$ to $f$.
\end{lemma}

\Cref{lemma delta well-defined} below generalizes \cite[Lemma 3.3, p. 293]{sae2} by showing that the assignment $[f \colon M^{n} \rightarrow \mathbb{R}] \mapsto [\sharp(M^{n})]$ defines a well-defined map $\delta_{k}^{n} \colon \mathcal{G}_{k}^{n} \rightarrow \mathcal{C}_{k-1}^{n}$ for any integer $1 < k < n$.
Then, it follows that $\delta_{k}^{n} \colon \mathcal{G}_{k}^{n} \rightarrow \mathcal{C}_{k-1}^{n}$ is a homomorphism because the existence of an orientation preserving diffeomorphism $\sharp(M \sqcup N) \cong \sharp(\sharp(M) \sqcup \sharp(N))$ implies that $\delta_{k}^{n}([f] + [g]) = [\sharp(M \sqcup N)] = [\sharp(\sharp(M) \sqcup \sharp(N))] = \delta_{k}^{n}([f]) + \delta_{k}^{n}([g])$ for any $k$-constrained Morse functions $f \colon M^{n} \rightarrow \mathbb{R}$ and $g \colon N^{n} \rightarrow \mathbb{R}$.

\begin{proposition}\label{lemma delta well-defined}
Suppose that $1 < k < n$.
Let $f \colon M^{n} \rightarrow \mathbb{R}$ and $g \colon N^{n} \rightarrow \mathbb{R}$ be $k$-constrained Morse functions which are oriented $k$-constrained generic bordant.
Then, there exists an oriented $(k-1)$-connective bordism $V^{n+1}$ from $M^{n}$ to $N^{n}$.
\end{proposition}

\begin{figure}[htbp]
  \centering
  
   \fbox{\begin{tikzpicture}[scale=0.2]
    \draw (0, 0) node[inner sep=0] {\includegraphics[width=0.65\textwidth]{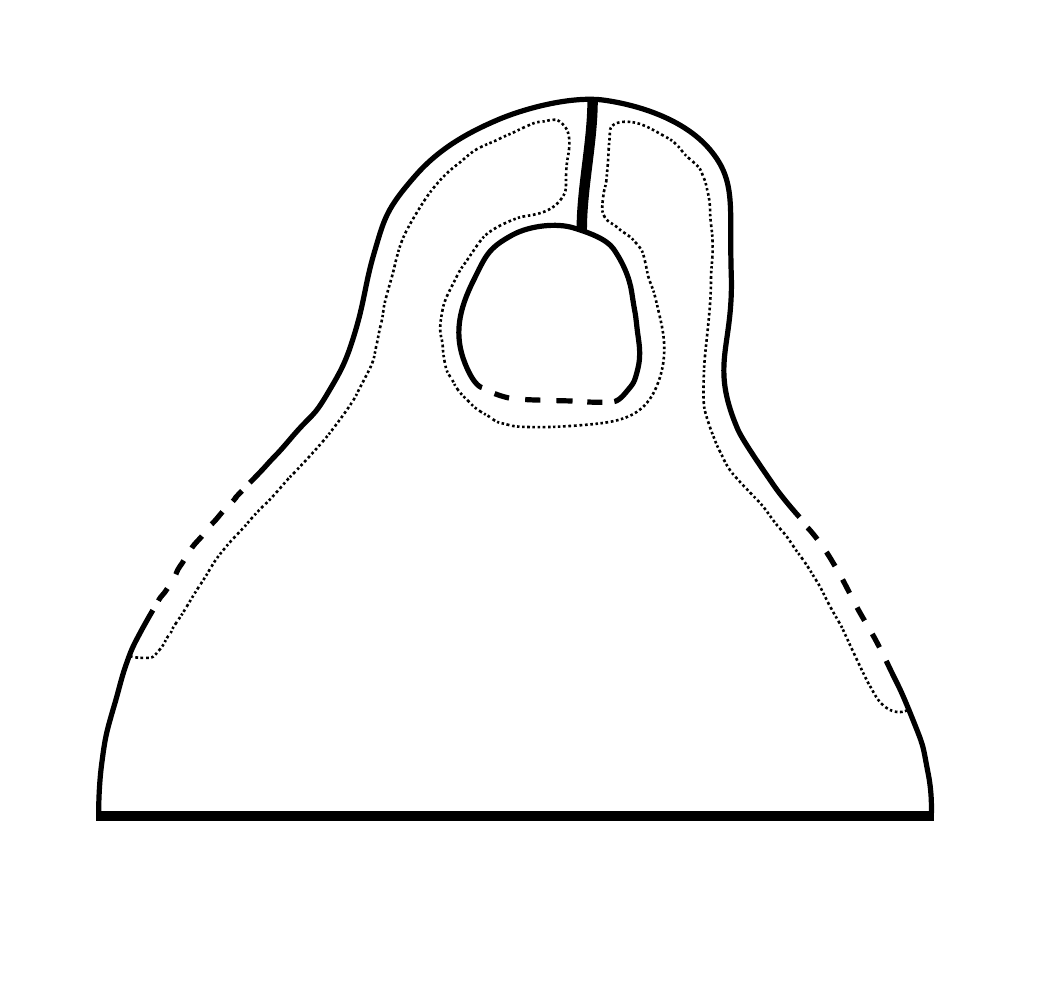}};
    \draw (-1, -15) node {$\pi_{H}(\sharp(M^{n} \sqcup - N^{n}))$};
    \draw (2.5, 17.5) node {$L_{i}$};
    \draw (-2, 2) node {$L$};
    \draw (0, -7) node {$U_{H}$};
\end{tikzpicture}}
 
  \caption{Stein factorization $\pi_{H} \colon U \rightarrow U_{H}$ of $H \colon U^{n+1} \rightarrow \mathbb{R}^{2}$.
  Exemplarily, only some $1$-handle $T_{i}$ is shown explicitly.
  The line segment $[0, 1] \cong L \subset U_{H}$ is chosen as indicated by the dotted line.}
  \label{fig:stein_factorization}
\end{figure}
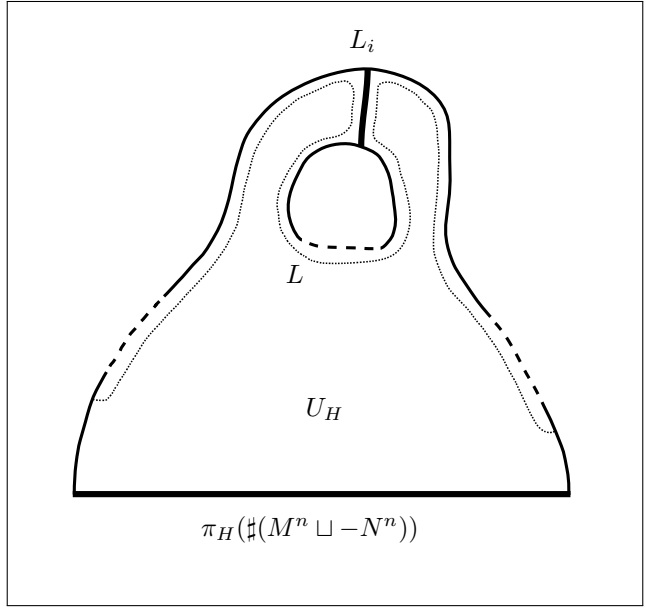

\begin{proof}
By assumption there exists an oriented $k$-constrained generic bordism from $f \sqcup -g \colon M^{n} \sqcup -N^{n} \rightarrow \mathbb{R}$ to $\emptyset \rightarrow \mathbb{R}$.
\Cref{lemma epsilon homomorphism} implies that there exist a $k$-constrained Morse function $h \colon \sharp(M^{n} \sqcup - N^{n}) \rightarrow \mathbb{R}$ and an oriented $k$-constrained generic bordism $H \colon U^{n+1} \rightarrow \mathbb{R}^{2}$ from $h$ to $\emptyset \rightarrow \mathbb{R}$.
By \Cref{stein factorization for fold maps with indefinite fold lines} the Stein factorization $U_{H}$ of $H$ (see \Cref{fig:stein_factorization}) is a connected oriented compact $2$-manifold with two corners, where we have assumed without loss of generality that $H$ is stable.
(In fact, $H$ can always be achieved to be stable by first choosing $h$ to have a single critical point on each level set, and then perturbing $H$ on the interior of $U^{n+1}$.)
Then, by the classification of compact surfaces, we may think of $U_{H}$ as a half disc with a finite number $r \geq 0$ of $1$-handles attached, say $T_{1}, \dots, T_{r}$.
Let $L_{1}, \dots, L_{r} \subset U^{n+1}$ be the line segments indicated in \Cref{fig:stein_factorization}.
We choose an oriented compact submanifold $V^{n+1}$ of $U^{n+1}$ with boundary $\partial V = \sharp(M^{n} \sqcup - N^{n}) \sqcup -P^{n}$, where $P^{n} \subset U^{n+1}$ is an oriented submanifold of codimension $1$ which is constructed as the preimage $P^{n} = \pi_{H}^{-1}(L)$ of a smooth line segment $L \subset U_{H}$ chosen as indicated in \Cref{fig:stein_factorization}.
Here, we have assumed that $L$ avoids images of cusps of $\pi_{H}$, and is transverse to $\pi_{H}(S(\pi_{H}))$.
Then, we have an orientation preserving diffeomorphism $P^{n} \cong \sharp(P_{1} \sqcup (-P_{1}) \sqcup \dots \sqcup P_{r} \sqcup (-P_{r}))$, where $P_{i} := \pi_{H}^{-1}(L_{i})$.
By construction, the Stein factorization of the restriction $F := H|_{V}$ is diffeomorphic to a rectangle, $V_{F} \cong [0, 1] \times [0, 1]$, in such a way that $\{0, 1\} \times [0, 1]$ corresponds to $\pi_{F}(\partial V)$, and $[0, 1] \times \{0, 1\}$ corresponds to $\pi_{F}(S)$, where $S$ denotes the definite fold locus of $F$.
Let $\rho \colon V_{F} \rightarrow [0, 1]$ denote the projection of $V_{F} \cong [0, 1] \times [0, 1]$ to the first factor.
Then, by slightly perturbing $\rho$ on the interior of $V_{F}$, we can achieve that the composition $\rho \circ \pi_{F} \colon V^{n+1} \rightarrow [0, 1]$ is a $k$-constrained Morse function with regular level sets $\sharp(M^{n} \sqcup N^{n})$ and $P^{n}$.
Consequently, $V^{n+1}$ is a $(k-1)$-connective bordism from $\sharp(M^{n} \sqcup N^{n})$ to $P^{n}$, and the claim follows.
\end{proof}

From now on, suppose that $n \geq 6$. \par\medskip

Let us show that the choice of $k$-constrained Morse functions $f \colon M^{n} \rightarrow \mathbb{R}$ on given $k$-connected oriented closed manifolds $M^{n}$ induces a well-defined map $\varepsilon_{k}^{n} \colon \mathcal{C}_{k}^{n} \rightarrow \mathcal{G}_{k}^{n}$, $[M^{n}] \mapsto [f \colon M^{n} \rightarrow \mathbb{R}]$.
Note that $\varepsilon_{k}^{n}$ will be a homomorphism by \Cref{lemma epsilon homomorphism}.
We distinguish between the cases that $k > n/2$, $k = n/2$ and $k < n/2$.

For $k > n/2$, note that $\mathcal{C}_{k}^{n} = \Theta_{n}$, the group of homotopy spheres, and $\mathcal{G}_{k}^{n} = \widetilde{\Gamma}(n, 1)$, the bordism group of special generic functions (see \Cref{G and M coincide for large k}).
Then, $\varepsilon_{k}^{n}$ coincides with the homomorphism
$$\widetilde{\Phi} \colon \Theta_{n} \rightarrow \widetilde{\Gamma}(n, 1)$$
introduced in the proof of \cite[Theorem 1.1, p. 288]{sae2} (compare \cite[p. 294]{sae2}), and is well-defined by \cite[Lemma 3.1, p. 291]{sae2} because $n \geq 6$.

For $k = n/2$ we have already noted in \Cref{connective bordism} that $\mathcal{C}_{n/2}^{n} = \Theta_{n}$, the group of homotopy spheres.
All indefinite critical points of an $n/2$-constrained Morse function $f \colon M^{n} \rightarrow \mathbb{R}$ are of index $n/2$, so there are none if $M^{n}$ is a homotopy sphere.
Hence, the desired homomorphism $\varepsilon_{n/2}^{n} \colon \mathcal{C}_{n/2}^{n} \rightarrow \mathcal{G}_{n/2}^{n}$ is the composition of the above homomorphism $\mathcal{C}_{n/2}^{n} = \Theta_{n} \stackrel{\widetilde{\Phi}}{\longrightarrow} \widetilde{\Gamma}(n, 1) = \mathcal{G}_{n/2+1}^{n}$ and the homomorphism $\mathcal{G}_{n/2+1}^{n} \rightarrow \mathcal{G}_{n/2}^{n}$.

Let $k < n/2$.
For $k = 1$, the map $\varepsilon_{1}^{n}$ is just the composition of the map $\mathcal{C}_{1}^{n} \rightarrow \Omega_{n}^{SO}$, $[M^{n}] \rightarrow [M^{n}]$, with the inverse of the isomorphism $\mathcal{G}_{1}^{n} \stackrel{\cong}{\longrightarrow} \Omega_{n}^{SO}$ described in \Cref{remark interpolation of generic constrained bordism}.
For $1 < k < n/2$, evaluation of the map $\varepsilon_{k}^{n} \colon \mathcal{C}_{k}^{n} \rightarrow \mathcal{G}_{k}^{n}$ on a given element of $\mathcal{C}_{k}^{n}$ depends a priori on the choice of a representative $M^{n}$, and on the choice of a $k$-constrained Morse function $f \colon M^{n} \rightarrow \mathbb{R}$.
Independence of choices follows from
the following result (compare \cite[Lemma 3.1, p. 291]{sae2}).

\begin{proposition}\label{existence of constrained bordisms}
Suppose that $1 < k < n/2$.
Let $W^{n+1}$ be an oriented $(k-1)$-connective bordism from $M_{0}^{n}$ to $M_{1}^{n}$, and let $g_{0} \colon M_{0}^{n} \rightarrow \mathbb{R}$ and $g_{1} \colon M_{1}^{n} \rightarrow \mathbb{R}$ be $k$-constrained Morse functions.
If $W^{n+1}$ is $k$-connected, then there exists an oriented $k$-constrained generic bordism $G \colon W^{n+1} \rightarrow \mathbb{R}^{2}$ from $g_{0}$ to $g_{1}$.
\end{proposition}

\begin{proof}
By the rearrangement theorem \cite[p. 44]{mil2} we can modify an arbitrarily chosen Morse function $f \colon W^{n+1} \rightarrow [-1/2, n+1+1/2]$ with regular level sets $M_{0}^{n} = f^{-1}(-1/2)$ and $M_{1}^{n} = f^{-1}(n+1+1/2)$ to be self-indexing, i.e., the Morse index of every critical point $c$ of $f$ is $f(c)$.
Following the proof of Smale's $h$-cobordism theorem \cite[Theorem 9.1, p. 107]{mil2}, we then modify $f$ in such a way that all Morse indices of critical points of $f$ are contained in the set $\{k+1, \dots, n-k\}$.
By \Cref{theorem handle extension} there exist for every $\lambda \in \{k+1, \dots, n-k\}$ $k$-constrained Morse functions $g_{\lambda}^{-} \colon f^{-1}(\lambda - 1/4) \rightarrow \mathbb{R}$ and $g_{\lambda}^{+} \colon f^{-1}(\lambda + 1/4) \rightarrow \mathbb{R}$, and an oriented $k$-constrained generic bordism $G_{\lambda} \colon f^{-1}([\lambda - 1/4, \lambda + 1/4]) \rightarrow \mathbb{R}^{2}$ from $g_{\lambda}^{-}$ to $g_{\lambda}^{+}$.
Furthermore, by \Cref{theorem cylinder bordism between constrained morse functions} there exist oriented $k$-constrained generic bordisms
$G_{0, k+1} \colon f^{-1}([-1/2, (k + 1) - 1/4]) \rightarrow \mathbb{R}^{2}$ from $g_{0}$ to $g_{k+1}^{-}$,
$G_{\lambda, \lambda+1} \colon f^{-1}([\lambda +1/4, (\lambda + 1) - 1/4]) \rightarrow \mathbb{R}^{2}$ from $g_{\lambda}^{+}$ to $g_{\lambda+1}^{-}$, $\lambda \in \{k+1, \dots n-k-1\}$, and
$G_{n-k, n+1} \colon f^{-1}([(n-k) + 1/4, n+1+1/2]) \rightarrow \mathbb{R}^{2}$ from $g_{n-k}^{+}$ to $g_{1}$, and the claim follows.
\end{proof}

From now on, let $1 < k < n$.
The composition $\delta_{k}^{n} \circ \varepsilon_{k}^{n}$ coincides by definition with the natural homomorphism $\mathcal{C}_{k}^{n} \rightarrow \mathcal{C}_{k-1}^{n}$, $[M^{n}] \mapsto [M^{n}]$.
Let us show that the composition $\varepsilon_{k-1}^{n} \circ \delta_{k}^{n}$ coincides with the natural homomorphism $\mathcal{G}_{k}^{n} \rightarrow \mathcal{G}_{k-1}^{n}$, $[f \colon M^{n} \rightarrow \mathbb{R}] \mapsto [f]$.
By definition, the composition $\varepsilon_{k-1}^{n} \circ \delta_{k}^{n}$ assigns to the class $[f] \in \mathcal{G}_{k}^{n}$ of a $k$-constrained Morse function the class $[g] \in \mathcal{G}_{k-1}^{n}$ represented by an arbitrarily chosen $(k-1)$-constrained Morse function $g \colon \sharp(M^{n}) \rightarrow \mathbb{R}$.
In view of \Cref{lemma epsilon homomorphism} it suffices to show that any two $(k-1)$-constrained Morse functions on $\sharp(M^{n})$ are oriented $(k-1)$-constrained generic bordant.
If $k-1 < n/2$, then this holds by \Cref{theorem cylinder bordism between constrained morse functions}.
If $k-1 \geq n/2$, then $\mathcal{G}_{k}^{n} = \widetilde{\Gamma}(n, 1)$, the bordism group of special generic functions (see \Cref{G and M coincide for large k}).
Hence, $f \colon M^{n} \rightarrow \mathbb{R}$ is a special generic function, $\sharp(M^{n})$ is a homotopy $n$-sphere, and the claim follows from \cite[Lemma 3.1, p. 291]{sae2} because $n \geq 6$.
\par\medskip

This completes the proof of \Cref{theorem constrained generic bordism and connective bordism}.

\begin{remark}\label{remark unoriented version for constrained generic bordism and connective bordism}
Unoriented versions $\mathcal{D}_{k}^{n}$ and $\mathcal{H}_{k}^{n}$ of the groups $\mathcal{C}_{k}^{n}$ and $\mathcal{G}_{k}^{n}$, can be defined in a straightforward way by forgetting about orientations in \Cref{definition connective bordism} and \Cref{definition constrained generic bordism}, respectively.
It is easy to show that the natural epimorphisms $\mathcal{C}_{k}^{n} \rightarrow \mathcal{D}_{k}^{n}$, $[M^{n}] \mapsto [M^{n}]$, has kernel $2 \mathcal{C}_{k}^{n}$.
Moreover, for $k > 1$ we can modify the argument of the proof of \Cref{lemma delta well-defined} in order to show that the epimorphism $\mathcal{G}_{k}^{n} \rightarrow \mathcal{H}_{k}^{n}$, $[f \colon M^{n} \rightarrow \mathbb{R}] \mapsto [f \colon M^{n} \rightarrow \mathbb{R}]$, has kernel $2 \mathcal{G}_{k}^{n}$ (compare \cite[Remark 3.4, p. 293]{sae2}).
For $k = 1$ we use the isomorphism $\mathcal{G}_{1}^{n} \stackrel{\cong}{\longrightarrow} \Omega_{n}^{SO}$ described in \Cref{remark interpolation of generic constrained bordism}, and the analogously defined isomorphism $\mathcal{H}_{1}^{n} \stackrel{\cong}{\longrightarrow} \Omega_{n}^{O}$ to the unoriented smooth bordism group to show that the homomorphism $\mathcal{G}_{1}^{n} \rightarrow \mathcal{H}_{1}^{n}$, $[f \colon M^{n} \rightarrow \mathbb{R}] \mapsto [f \colon M^{n} \rightarrow \mathbb{R}]$, has kernel $2 \mathcal{G}_{1}^{n}$.
(In fact, note that the homomorphism $\Omega_{n}^{SO} \rightarrow \Omega_{n}^{O}$, $[M^{n}] \rightarrow [M^{n}]$, has kernel $2 \Omega_{n}^{SO}$.)
Hence, all the statements of \Cref{theorem constrained generic bordism and connective bordism} carry over to unoriented bordism groups.
\end{remark}

\section{Proof of \Cref{theorem computation of oriented bordism of constrained Morse functions}}\label{section proof constrained bordism}

Fix integers $n \geq 4$ and $1 < k \leq n/2$.
We present the proof of \Cref{theorem computation of oriented bordism of constrained Morse functions} in a sequence of lemmas.

\begin{lemma}\label{lemma surjectivity of beta}
The homomorphism
\begin{displaymath}
\beta_{k}^{n} \colon \mathcal{M}_{k}^{n} \rightarrow \mathcal{G}_{k}^{n} \oplus \mathbb{Z}^{\oplus \lfloor n/2 \rfloor - k}, \quad [f \colon M^{n} \rightarrow \mathbb{R}] \mapsto ([f], \Phi_{k}^{n}([f])),
\end{displaymath}
is surjective, where $\Phi_{k}^{n} \colon \mathcal{M}_{k}^{n} \rightarrow \mathbb{Z}^{\oplus \lfloor n/2 \rfloor - k}$ is defined in \Cref{section cusps}.
\end{lemma}

\begin{proof}
Given an element $([g \colon M^{n} \rightarrow \mathbb{R}], c_{1}, \dots, c_{\lfloor n/2\rfloor - k}) \in \mathcal{G}_{k}^{n} \oplus \mathbb{Z}^{\oplus \lfloor n/2 \rfloor - k}$, we can use the same argument as in \cite[p. 220]{ike} to modify $g$ iteratively by introducing pairs of critical points of successive Morse indices in order to produce a $k$-constrained Morse function $f \colon M^{n} \rightarrow \mathbb{R}$ which satisfies $\Phi_{k}^{n}([f]) = (c_{1}, \dots, c_{\lfloor n/2\rfloor - k})$.
Since $[f] = [g]$ in $\mathcal{G}_{k}^{n}$ by means of a generic homotopy that realizes the sequence of births of critical point pairs, we obtain $\beta_{k}^{n}([f \colon M^{n} \rightarrow \mathbb{R}]) = ([g], c_{1}, \dots, c_{\lfloor n/2\rfloor - k})$.
\end{proof}

The following lemma completes the proof of part $(i)$.

\begin{lemma}
If $n$ is even, then $\beta_{k}^{n}$ is injective.
\end{lemma}

\begin{proof}
Suppose that $\beta_{k}^{n}([f]) = 0 \in \mathcal{G}_{k}^{n} \oplus \mathbb{Z}^{\oplus \lfloor n/2 \rfloor - k}$ for some $[f \colon M^{n} \rightarrow \mathbb{R}] \in \mathcal{M}_{k}^{n}$.
Then, there exists an oriented $k$-constrained generic bordism $G \colon W^{n+1} \rightarrow \mathbb{R}^{2}$ from $f$ to $f_{\emptyset}$.
Since $\Phi_{k}^{n}([f]) = 0$, \Cref{proposition creating and eliminating cusps} implies that $[f] = [f_{\emptyset}] = 0 \in \mathcal{M}_{k}^{n}$.
\end{proof}

In order to studying the kernel of $\beta_{k}^{n}$ when $n$ is odd, we introduce a homomorphism $\alpha_{k}^{n} \colon \mathbb{Z}/2 \rightarrow \mathcal{M}_{k}^{n}$ as follows.

\begin{lemma}\label{lemma definition of homomorphism alpha}
Suppose that $n$ is odd, and let $l := (n-1)/2$.
If $f, g \colon S^{n} \rightarrow \mathbb{R}$ are Morse functions with exactly $4$ critical points whose Morse indices form the set $\{0, l, l+1, n\}$, then $[f] = [g] \in \mathcal{M}_{k}^{n}$.
Hence, there is a well-defined homomorphism
\begin{displaymath}
\alpha_{k}^{n} \colon \mathbb{Z}/2 \rightarrow \mathcal{M}_{k}^{n}, \qquad \overline{1} \mapsto [f_{\alpha} \colon S^{n} \rightarrow \mathbb{R}],
\end{displaymath}
where $f_{\alpha} \colon S^{n} \rightarrow \mathbb{R}$ denotes a fixed Morse function with the above properties.
\end{lemma}

\begin{proof}
By \Cref{theorem cylinder bordism between constrained morse functions} there exists an oriented $l$-constrained generic bordism $G \colon S^{n} \times [0, 1] \rightarrow \mathbb{R}^{2}$ from $f$ to $g$.
Moreover, \Cref{proposition euler characteristic and cusps} implies that the number of cusps of $G$ is even.
Hence, the claim that $[f] = [g] \in \mathcal{M}_{k}^{n}$ follows from \Cref{proposition creating and eliminating cusps}.
As $S^{n}$ admits an orientation reversing automorphism, we may choose $g = -f$, and obtain $2 [f] = 0$.
\end{proof}

\begin{lemma}\label{lemma image of alpha equals kernel of beta}
If $n$ is odd, then $\operatorname{im} \alpha_{k}^{n} = \operatorname{ker} \beta_{k}^{n}$.
\end{lemma}

\begin{proof}
Note that $\widetilde{\Psi}^{k}([f_{\alpha}]) = 0 \in \mathcal{G}_{k}^{n}$ and $\Phi_{k}^{n}([f_{\alpha}]) = 0 \in \mathbb{Z}^{\oplus \lfloor n/2 \rfloor - k}$, and thus $\operatorname{im} \alpha_{k}^{n} \subset \operatorname{ker} \beta_{k}^{n}$.
Conversely, suppose that $\beta_{k}^{n}([f]) = 0 \in \mathcal{G}_{k}^{n} \oplus \mathbb{Z}^{\oplus \lfloor n/2 \rfloor - k}$ for some $[f \colon M \rightarrow \mathbb{R}] \in \mathcal{M}_{k}^{n}$.
Then, there exists an oriented $k$-constrained generic bordism $G \colon W^{n+1} \rightarrow \mathbb{R}^{2}$ from $f$ to $f_{\emptyset}$.
Moreover, by means of \Cref{theorem cylinder bordism between constrained morse functions}, we can construct an oriented $k$-constrained generic bordism $G_{\alpha} \colon D^{n+1} \rightarrow \mathbb{R}^{2}$ from $f_{\alpha}$ to $f_{\emptyset}$.
By \Cref{proposition euler characteristic and cusps} exactly one of the oriented $k$-constrained generic bordisms $G$ and $G \sqcup G_{\alpha}$, say $G_{0}$, has an even number of cusps.
Hence, \Cref{proposition creating and eliminating cusps} implies that $0 = [f_{0}] = [f] + m [f_{\alpha}] \in \mathcal{M}_{k}^{n}$ for suitable $m \in \{0, 1\}$, and we conclude that $\operatorname{ker} \beta_{k}^{n} \subset \operatorname{im} \alpha_{k}^{n}$.
\end{proof}

\Cref{existence of a splitting} below completes the proof of part $(ii)$.
For this purpose, recall from \cite[Definition 2.5, p. 214]{ike} that there is for $n \equiv 1 \, \operatorname{mod} 4$ a well-defined homomorphism $\Lambda \colon \mathcal{M}_{1}^{n} \rightarrow \mathbb{Z}/2$ which assigns to $[f] \in \mathcal{M}_{1}^{n}$ the element $\Lambda([f]) = \sigma(f) - \sigma(M^{n}; \mathbb{Q}) \in \mathbb{Z}/2$, where $\sigma(f) = \sum_{\lambda = 0}^{(n-1)/2}C_{\lambda}(f)$, and $\sigma(M^{n}; \mathbb{Q})$ denotes the Kervaire semi-characteristic of $M^{n}$ over $\mathbb{Q}$ (see \cite{ker}).
Composition with the natural homomorphism $\mathcal{M}_{k}^{n} \rightarrow \mathcal{M}_{1}^{n}$ yields a homomorphism $\Lambda^{n}_{k} \colon \mathcal{M}_{k}^{n} \rightarrow \mathbb{Z}/2$ which turns out to be a splitting of $\alpha_{k}^{n}$ for $n \equiv 1 \, \operatorname{mod} 4$.

\begin{lemma}\label{existence of a splitting}
For $n \equiv 1 \, \operatorname{mod} 4$ we have $\Lambda^{n}_{k} \circ \alpha_{k}^{n} = \operatorname{id}_{\mathbb{Z}/2}$.
\end{lemma}

\begin{proof}
It suffices to note that $\Lambda^{n}_{k}([f_{\alpha}]) = \overline{1} \in \mathbb{Z}/2$, which follows from $\sigma(f_{\alpha}) = 2$ and $\sigma(S^{n}; \mathbb{Q}) = 1$.
\end{proof}

In \Cref{injectivity of alpha depending on the sequence} below we prove the remaining parts $(iii)$ and $(iv)$ by characterizing injectivity of $\alpha_{k}^{n}$ for $n \equiv 3 \, \operatorname{mod} 4$.
For this purpose, we introduce the required sequence $\kappa_{1}, \kappa_{2}, \dots$ of positive integers in \Cref{definition sequence existence of odd euler characteristic and constrained generic map} below.
Note that $\mathbb{C}P^{2i}$ is for any integer $i \geq 1$ a closed $4i$-manifold with odd Euler characteristic, and any generic map $\mathbb{C}P^{2i} \rightarrow \mathbb{R}^{2}$ defines an oriented $1$-constrained generic bordism from $f_{\emptyset} \colon \emptyset \rightarrow \mathbb{R}$ to $f_{\emptyset} \colon \emptyset \rightarrow \mathbb{R}$.

\begin{definition}\label{definition sequence existence of odd euler characteristic and constrained generic map}
For every integer $i \geq 1$ let $\kappa_{i}$ be the greatest integer $k \geq 1$ for which there exists an oriented $k$-constrained generic bordism $V^{4i} \rightarrow \mathbb{R}^{2}$ from $f_{\emptyset} \colon \emptyset \rightarrow \mathbb{R}$ to $f_{\emptyset} \colon \emptyset \rightarrow \mathbb{R}$ such that $V^{4i}$ has odd Euler characteristic (or, equivalently, odd signature).
\end{definition}

\begin{lemma}\label{injectivity of alpha depending on the sequence}
Suppose that $n \equiv 3 \, \operatorname{mod} 4$.
Then, $\alpha_{k}^{n}$ is injective if and only if $k > \kappa_{(n+1)/4}$.
\end{lemma}

\begin{proof}
As in the proof of \Cref{lemma image of alpha equals kernel of beta} we can construct an oriented $k$-constrained generic bordism $G_{\alpha} \colon D^{n+1} \rightarrow \mathbb{R}^{2}$ from $f_{\alpha}$ to $f_{\emptyset}$.

``$\Leftarrow$''.
Let $k > \kappa_{(n+1)/4}$.
If we suppose that $[f_{\alpha}] = 0 \in \mathcal{M}_{k}^{n}$, then there exists an oriented $k$-constrained bordism $F \colon W^{n+1} \rightarrow \mathbb{R}^{2}$ from $f_{\emptyset}$ to $f_{\alpha}$.
From \Cref{proposition euler characteristic and cusps} we conclude that $\chi(W)$ is even because $F$ has no cusps and $f_{\alpha}$ has $4$ critical points.
Hence, $V^{n+1} := W^{n+1} \cup_{S^{n}} D^{n+1}$ is an oriented closed manifold with odd Euler characteristic.
Now $F$ and $G_{\alpha}$ glue to an oriented $k$-constrained generic bordism $V^{n+1} \rightarrow \mathbb{R}^{2}$ from $f_{\emptyset}$ to $f_{\emptyset}$, which contradicts the assumption that $k > \kappa_{(n+1)/4}$.

``$\Rightarrow$''.
Suppose that $k \leq \kappa_{(n+1)/4}$, and let $V^{n+1}$ be a closed manifold with odd Euler characteristic which admits an oriented $k$-constrained generic bordism $G_{1} \colon V^{n+1} \rightarrow \mathbb{R}^{2}$ from $f_{\emptyset}$ to $f_{\emptyset}$.
Note that both $G_{\alpha}$ and $G_{1}$ have an odd number of cusps by \Cref{proposition euler characteristic and cusps}.
Hence, the oriented $k$-constrained generic bordism $G_{\alpha} \sqcup G_{1} \colon D^{n+1} \sqcup V^{n+1} \rightarrow \mathbb{R}^{2}$ from $f_{\alpha}$ to $f_{\emptyset}$ has an even number of cusps.
Therefore, $[f_{\alpha}] = 0 \in \mathcal{M}_{k}^{n}$ by \Cref{proposition creating and eliminating cusps}.
\end{proof}

\begin{remark}
We do not know if the short exact sequence in \Cref{theorem computation of oriented bordism of constrained Morse functions}$(iv)$ splits.
\end{remark}

Finally, the sequences $\kappa_{1}, \kappa_{2}, \dots$ and $\gamma_{1}, \gamma_{2}, \dots$ are related to each other as follows.
The inequality $\gamma_{i} \leq \kappa_{i}$ is (for $\gamma_{i} > 1$) a consequence of \Cref{existence of constrained bordisms}.
Moreover, the inequality $\kappa_{i} \leq \gamma_{i} + 1$ follows (for $\kappa_{i} > 1$) from \Cref{lemma delta well-defined}, where one has to take care of the parity of the Euler characteristic.
\par\medskip

This completes the proof of \Cref{theorem computation of oriented bordism of constrained Morse functions}.

\begin{remark}
Note that $\kappa_{i} < 2i$ for all $i \geq 1$ (where we have used \Cref{remark no cusps for 2k = n} and \Cref{proposition euler characteristic and cusps} to exclude the case that $\kappa_{i} = 2i$).
Hence, for $\gamma_{i} = 2i-1$ (which happens (at least) for $i = 1, 2 ,4$) we have $\kappa_{i} = \gamma_{i}$.
Moreover, $\gamma_{i} \not\equiv 2, 4, 5, 6 \, \operatorname{mod} 8$ implies that $\kappa_{i} \not\equiv 5, 6 \, \operatorname{mod} 8$ because $\gamma_{i} \in \{\kappa_{i}-1, \kappa_{i}\}$.
\end{remark}

\begin{remark}\label{remark unoriented version for bordism of constrained Morse functions}
An unoriented version $\mathcal{N}_{k}^{n}$ of the group $\mathcal{M}_{k}^{n}$ can be defined in a straightforward way by forgetting orientations in \Cref{definition constrained bordism}.
Then, one obtains a version of \Cref{theorem computation of oriented bordism of constrained Morse functions} involving the unoriented bordism groups $\mathcal{N}_{k}^{n}$ and $\mathcal{H}_{k}^{n}$ (see \Cref{remark unoriented version for constrained generic bordism and connective bordism}), where \Cref{definition sequence existence of odd euler characteristic and constrained generic map} and \Cref{definition sequence gamma} have to be modified appropriately.
\end{remark}

\section{Detecting Exotic Kervaire Spheres}\label{kervaire spheres}\label{application to Kervaire spheres}

We discuss how bordism groups of constrained Morse functions are capable of detecting exotic Kervaire spheres in certain high dimensions.
Recall that Kervaire spheres are a concrete family of homotopy spheres that can be obtained from a plumbing construction as follows (see \cite[p. 162]{law}).
The unique Kervaire sphere $\Sigma_{K}^{n}$ of dimension $n = 4k+1$ can be defined as the boundary of the parallelizable $(4k+2)$-manifold given by plumbing together two copies of the tangent disc bundle of $S^{2k+1}$.
In \Cref{theorem kervaire spheres} we will need to impose stronger conditions on the dimension $n$.
These originate from a result of Stolz \cite{stol} on highly connected bordisms that is used in our argument, and we do not know if they can be eliminated.

Let $\Theta_{n}$ denote the group of homotopy $n$-spheres.
Recall that $\Theta_{n}$ consists of $h$-cobordism classes of oriented differentiable homotopy $n$-spheres, and its group structure is induced by forming the oriented connected sum of homotopy spheres.
The group of homotopy spheres has been introduced and studied by Kervaire and Milnor \cite{KM}, who showed that for $n \geq 5$, $\Theta_{n}$ is a finite abelian group.
As remarked in \cite[p. 505]{KM}, $\Theta_{n}$ classifies for $n \geq 5$ oriented diffeomorphism classes of oriented closed $n$-manifolds homeomorphic to $S^{n}$, where non-trivial classes are usually called \emph{exotic spheres}.

Let $bP_{n+1} \subset \Theta_{n}$ denote the subgroup of those homotopy $n$-spheres that can be realized as the boundary of a parallelizable compact manifold (see \cite[p. 510]{KM}).
For instance, we have $[\Sigma_{K}^{4k+1}] \in bP_{4k+2}$ by construction of the Kervaire spheres, and $bP_{n+1} = 0$ whenever $n$ is even by \cite[Theorem 5.1, p. 512]{KM}.
The following result is part of the classification theorem of homotopy spheres (see \cite[Theorem 6.1, pp. 123f]{luck}).

\begin{theorem}\label{theorem classification of homotopy spheres}
Suppose that $n = 4k+1$ for some integer $k \geq 1$.
Then, $bP_{n+1} = \{[S^{n}], [\Sigma_{K}^{n}]\}$, where $\Sigma_{K}^{n}$ denotes the unique Kervaire sphere of dimension $n$.
Moreover, $bP_{n+1} \cong \mathbb{Z}/2$ whenever $n+3 \notin \{2^{1}, 2^{2}, 2^{3}, \dots\}$, whereas $bP_{n+1} = 0$ for $n \in\{5, 13, 29, 61\}$.
We have $\Theta_{n}/bP_{n+1} \cong \operatorname{coker}J_{n}$, where $J_{n} \colon \pi_{n}(SO) \rightarrow \pi_{n}^{s}$ denotes the stable $J$-homomorphism.
\end{theorem}

The application of our main results to Kervaire spheres is as follows.

\begin{theorem}\label{theorem kervaire spheres}
Suppose that $n \geq 237$ and $n \equiv 13 \, \operatorname{mod} 16$.
Let $l := (n-1)/2$.
Then, for any exotic $n$-sphere $\Sigma^{n}$ the following statements are equivalent:
\begin{enumerate}[$(i)$]
\item $\Sigma^{n}$ is diffeomorphic to the Kervaire $n$-sphere $\Sigma^{n}_{K}$.
\item $\Sigma^{n}$ admits an $l$-constrained Morse function which represents $0 \in \mathcal{M}_{l}^{n}$.
\end{enumerate}

\end{theorem}

\begin{remark}
There are infinitely many dimensions $n \equiv 13 \, \operatorname{mod} 16$ for which $bP_{n+1} \neq \Theta_{n}$, that is, $\Theta_{n}$ contains exotic spheres different from the Kervaire sphere.
In fact, by \Cref{theorem classification of homotopy spheres} it suffices to show that $\operatorname{coker}J_{n}$ is non-trivial for infinitely many such $n$.
For this purpose, note that for $n \equiv 5 \, \operatorname{mod} 8$ the domain of the stable $J$-homomorphism $J_{n} \colon \pi_{n}(SO) \rightarrow \pi_{n}^{s}$ satisfies $\pi_{n}(SO) = 0$, so that $\operatorname{coker}J_{n} = \pi_{n}^{s}$ (see the proof of Theorem 3.1 in \cite[p. 508]{KM}).
Now, we claim that $\pi_{n}^{s} \neq 0$ whenever $n = 2(p+1)(p-1)-3$ for an odd prime $p$.
Indeed, setting $(r,s) = (1,0)$, we can write any such $n$ in the form $n = 2(rp+s+1)(p-1)-2(r-s)-1$, where $0 \leq s < r \leq p-1$ and $r-s \neq p-1$.
Hence, it follows from \cite[Theorem B, p. 191]{toda} that the $p$-primary component of $\pi_{n}^{s}$ is $\mathbb{Z}_{p}$.

\end{remark}

\begin{proof}[Proof (of \Cref{theorem kervaire spheres})]
By composition of the natural homomorphism $\mathcal{C}_{l+1}^{n} \rightarrow \mathcal{C}_{l}^{n}$ with the homomorphisms of \Cref{theorem constrained generic bordism and connective bordism} we can define homomorphisms
\begin{align*}
c_{l}^{n} &\colon \Theta_{n} = \mathcal{C}_{l+1}^{n} \rightarrow \mathcal{C}_{l}^{n}, \quad [\Sigma^{n}] \mapsto [\Sigma^{n}], \\
g _{l}^{n} &\colon \Theta_{n} \stackrel{c_{l}^{n}}{\longrightarrow} \mathcal{C}_{l}^{n} \stackrel{\varepsilon_{l}^{n}}{\longrightarrow} \mathcal{G}_{l}^{n}, \quad [\Sigma^{n}] \mapsto [f], \\
c_{l-1}^{n} &\colon \Theta_{n} \stackrel{g _{l}^{n}}{\longrightarrow} \mathcal{G}_{l}^{n} \stackrel{\delta_{l}^{n}}{\longrightarrow} \mathcal{C}_{l-1}^{n}, \quad [\Sigma^{n}] \mapsto [\Sigma^{n}],
\end{align*}
where $f \colon \Sigma^{n} \rightarrow \mathbb{R}$ denotes an arbitrarily chosen $l$-constrained Morse function in the definition of $g _{l}^{n}$.
Denote the kernels of $c_{l}^{n}$, $g_{l}^{n}$ and $c_{l-1}^{n}$ by $C_{l}^{n}$, $G_{l}^{n}$ and $C_{l-1}^{n}$, respectively.
Then, by construction, $C_{l}^{n} \subset G_{l}^{n} \subset C_{l-1}^{n}$.

Note that \Cref{proposition bott periodicity} implies that $C_{l}^{n} = C_{l-1}^{n}$ because $l \equiv 6 \, \operatorname{mod} 8$.
Thus, we have shown that $G_{l}^{n} = C_{l}^{n}$.
Furthermore, the inclusion $bP_{n+1} \subset C_{l}^{n}$ holds since by \cite[Theorem 3, p. 49]{mil} any parallelizable compact smooth manifold $W^{m}$ of dimension $m = n+1$ can be made $l = (\lfloor m/2 \rfloor - 1)$-connected by a finite sequence of surgeries without changing $\partial W$.
Conversely, by a theorem of Stolz (see \cite[Theorem B$(ii)$, p. XIX]{stol}), the inclusion $C_{l}^{n} \subset bP_{n+1}$ holds because $m := n +1$ is by assumption of the form $m = 2k+d$ for $d = 0$ and some odd integer $k \geq 113$.
All in all, we have shown that $G_{l}^{n} = C_{l}^{n} = bP_{n+1} = \{[S^{n}], [\Sigma_{K}^{n}]\}$, where the last equality is taken from \Cref{theorem classification of homotopy spheres}.

Thus, for an exotic $n$-sphere $\Sigma^{n}$ statement $(i)$ holds if and only if $[\Sigma^{n}] \in G_{l}^{n}$, that is, $[f] = 0 \in \mathcal{G}_{l}^{n}$ for some (and hence, any) $l$-constrained Morse function $f \colon \Sigma^{n} \rightarrow \mathbb{R}$.
Equivalently, by \Cref{theorem computation of oriented bordism of constrained Morse functions}$(ii)$, $\Sigma^{n}$ satisfies statement $(ii)$.
(In fact, by \Cref{theorem computation of oriented bordism of constrained Morse functions}$(ii)$ and \Cref{existence of a splitting} we have an isomorphism $\mathcal{M}_{l}^{n} \cong \mathcal{G}_{l}^{n} \oplus \mathbb{Z}/2$ given by $[f \colon M^{n} \rightarrow \mathbb{R}] \mapsto ([f], \Lambda^{n}_{l}([f]))$.
If $\Sigma^{n}$ admits an $l$-constrained Morse function $f$ which represents $0 \in \mathcal{M}_{l}^{n}$, then obviously $[f] = 0 \in \mathcal{G}_{l}^{n}$.
Conversely, suppose that $\Sigma^{n}$ admits an $l$-constrained Morse function $f$ which represents $0 \in \mathcal{G}_{l}^{n}$.
If necessary, we modify $f$ by introducing an additional pair of Morse critical points of subsequent indices $l$ and $l+1$ in order to adjust the parity of $\sigma(f) = \sum_{\lambda = 0}^{l}C_{\lambda}(f)$ in such a way that $\Lambda^{n}_{l}([f]) = \sigma(f) - \sigma(\Sigma^{n}; \mathbb{Q}) = 0 \in \mathbb{Z}/2$.
Then, $[f] = 0 \in \mathcal{M}_{l}^{n}$.)
\end{proof}

\begin{remark}\label{remark subgroup filtrations group of homotopy spheres}
The groups $C_{l}^{n}$ and $G_{l}^{n}$ introduced in the proof of \Cref{theorem kervaire spheres} are part of natural subgroup filtrations $C^{n}_{\lfloor n/2\rfloor} \subset \dots \subset C_{1}^{n}$ and $G^{n}_{\lfloor n/2\rfloor} \subset \dots \subset G_{1}^{n}$ of $\Theta_{n}$ which can be defined for any $n \geq 6$ as follows.
Using \Cref{theorem constrained generic bordism and connective bordism}, we define for $1 \leq k \leq \lfloor n/2\rfloor$ the homomorphisms
\begin{align*}
c_{k}^{n} &\colon \Theta_{n} = \mathcal{C}_{\lfloor n/2\rfloor+1}^{n} \rightarrow \mathcal{C}_{k+1}^{n} \stackrel{\varepsilon_{k+1}^{n}}{\longrightarrow} \mathcal{G}_{k+1}^{n} \stackrel{\delta_{k+1}^{n}}{\longrightarrow} \mathcal{C}_{k}^{n}, \quad [\Sigma^{n}] \mapsto [\Sigma^{n}], \\
g _{k}^{n} &\colon \Theta_{n} \stackrel{c_{k}^{n}}{\longrightarrow} \mathcal{C}_{k}^{n} \stackrel{\varepsilon_{k}^{n}}{\longrightarrow} \mathcal{G}_{k}^{n}, \quad [\Sigma^{n}] \mapsto [f],
\end{align*}
where $f \colon \Sigma^{n} \rightarrow \mathbb{R}$ denotes an arbitrarily chosen $k$-constrained Morse function in the definition of $g _{k}^{n}$.
If $C_{k}^{n}$ and $G_{k}^{n}$ denote the kernels of $c_{k}^{n}$ and $g _{k}^{n}$, respectively, then $C_{k}^{n} \subset G_{k}^{n}$ for $1 \leq k \leq \lfloor n/2\rfloor$, and $G_{k}^{n} \subset C_{k-1}^{n}$ for $2 \leq k \leq \lfloor n/2\rfloor$ (compare \cite[Theorem 10.1.3, p. 243]{wra}).
Analogously to the proof of \Cref{theorem kervaire spheres} we can use \cite[Theorem 3, p. 49]{mil} to show that $bP_{n+1} \subset C^{n}_{\lfloor n/2 \rfloor}$.

In general, we do not know whether the resulting filtrations $C_{k}^{n}$ and $G_{k}^{n}$ of $\Theta_{n}$ coincide or not.
\end{remark}

\bibliographystyle{amsplain}

\begin{thebibliography}{10}

\bibitem{ban} M. Banagl, \emph{Positive topological quantum field theories}, Quantum Topology \textbf{6} (2015), no. 4, 609--706.
\bibitem{burder} O. Burlet, G. de Rham, \textit{Sur certaines applications g\'{e}n\'{e}riques d'une vari\'{e}t\'{e} close \`{a} trois dimensions dans le plan}, Enseign. Math. \textbf{20} (1974), 275--292.
\bibitem{cer} J. Cerf, \emph{La stratification naturelle des espaces de fonctions diff\'{e}rentiables r\'{e}elles et le th\'{e}or\`{e}me de la pseudo-isotopie},
Inst. Hautes \'{E}tudes Sci. Publ. Math. \textbf{39} (1970), 5--173.
\bibitem{cieeli} K. Cieliebak, Y. Eliashberg, \emph{From Stein to Weinstein and back. Symplectic geometry of affine complex manifolds}, American Mathematical Society Colloquium Publications \textbf{59} (2012).
\bibitem{dovschu} K.H. Dovermann, R. Schultz, \emph{Equivariant surgery theories and their periodicity properties}, Lecture Notes in Mathematics 1443, Springer-Verlag Berlin Heidelberg, 1990.
\bibitem{eli} J.M. Eliashberg, \emph{Surgery of singularities of smooth mappings}, Math. USSR. Izv. \textbf{6} (1972), 1302--1326.
\bibitem{gaykir} D.T. Gay, R. Kirby, \emph{Indefinite Morse 2-functions; broken fibrations and generalizations}, Geom. Topol. \textbf{19} (2015), 2465--2534.
\bibitem{hatwag} A. Hatcher, J. Wagoner, \emph{Pseudo-isotopies of compact manifolds}, Ast\'{e}risque \textbf{6}, Soc. Math. de France (1973).
\bibitem{hirsmale} M.W. Hirsch, \emph{Immersions of manifolds}, Trans. Amer. Math. Soc. \textbf{93} (1959), 242--276.
\bibitem{hir} M.W. Hirsch, \emph{Differential topology}, Grad. Texts in Math. \textbf{33}, Springer Verlag, 1976.
\bibitem{ike} K. Ikegami, \emph{Cobordism group of Morse functions on manifolds}, Hiroshima Math. J. \textbf{34} (2004), 211--230.
\bibitem{ikesae} K. Ikegami, O. Saeki, \emph{Cobordism group of Morse functions on surfaces}, J. Math. Soc. Japan \textbf{55} (2003), 1081--1094.
\bibitem{kal} B. Kalm\'{a}r, \emph{Pontryagin-Thom-Sz\H{u}cs type construction for non-positive codimensional singular maps with prescribed singular fibers}, The second Japanese-Australian Workshop on Real and Complex Singularities, RIMS K\^{o}ky\^{u}roku \textbf{1610} (2008), 66--79.
\bibitem{ker} M. Kervaire, \emph{Courbure int\'{e}grale g\'{e}n\'{e}ralis\'{e}e et homotopie}, Math. Ann. \textbf{131} (1956), 219--252.
\bibitem{KM} M. Kervaire, J. W. Milnor, \emph{Groups of homotopy spheres: I}, Ann. of Math. \textbf{77} (1963), 504--537.
\bibitem{kita} N. Kitazawa, \emph{Fold maps with singular value sets of concentric spheres}, Hokkaido Math. J. \textbf{43}, no. 3 (2014), 327--359.
\bibitem{kobsae} M. Kobayashi, O. Saeki, \emph{Simplifying stable mappings into the plane from a global viewpoint}, Transactions of the American Mathematical Society \textbf{348}, no. 7 (1996), 2607--2636.
\bibitem{law} H. B. Lawson, M.-L. Michelsohn, \emph{Spin Geometry}, Princeton Math. Series 38, Princeton University Press, (1989).
\bibitem{lev} H.I. Levine, \emph{Elimination of cusps}, Topology \textbf{3}, suppl. 2 (1965), 263--295.
\bibitem{luck} W. L\"{u}ck, \emph{A basic introduction to surgery theory}, version: October 27, 2004, \url{http://131.220.77.52/lueck/data/ictp.pdf}.
\bibitem{mil} J.W. Milnor, \emph{A procedure for killing homotopy groups of differentiable manifolds}, Symposia in Pure Math. \textbf{III}, American Mathematical Society (1961), 39--55.
\bibitem {mil2} J.W. Milnor, \emph{Lectures on the h-cobordism theorem}, Math. Notes, Princeton Univ. Press, Princeton, NJ, 1965.
\bibitem{och} S. Ochanine, \emph{Signature modulo 16, invariants de Kervaire g\'{e}n\'{e}ralis\'{e}s et nombres caract\'{e}ristiques dans la $K$-th\'{e}orie r\'{e}elle}, M\'{e}moires de la S. M. F. deuxi\`{e}me s\'{e}rie, tome \textbf{5} (1981), 1--142.
\bibitem{perl} N. Perlmutter, \emph{Cobordism categories and parametrized Morse theory}, \url{http://arxiv.org/abs/1703.01047v2}.
\bibitem{rim} R. Rim\'{a}nyi, A. Sz\H{u}cs, \emph{Pontrjagin-Thom-type construction for maps with singularities}, Topology \textbf{37} (1998), 1177--1191.
\bibitem{sad} R. Sadykov, \emph{Bordism groups of special generic mappings}, Proc. Amer. Math. Soc., \textbf{133} (2005), 931--936.
\bibitem{sae2} O. Saeki, \emph{Cobordism groups of special generic functions and groups of homotopy spheres}, Japan. J. Math. (N. S.) \textbf{28} (2002), 287--297.
\bibitem{sae3} O. Saeki, \emph{Topology of manifolds and global theory of singularities}, RIMS K\^{o}ky\^{u}roku Bessatsu \textbf{B55} (2016), 185--203.
\bibitem{stol} S. Stolz, \emph{Hochzusammenh\"{a}ngende Mannigfaltigkeiten und ihre R\"{a}nder}, Lecture Notes in Mathematics 1116, Springer-Verlag (1985).
\bibitem{sto} R.E. Stong, \emph{Notes on cobordism theory}, Princeton University Press, 1968.
\bibitem{thom} R. Thom, \emph{Quelques propri\'{e}t\'{e}s globales des vari\'{e}t\'{e}s diff\'{e}rentiables}, Comment. Math. Helv. \textbf{28} (1954),17--86.
\bibitem{toda} H. Toda, \emph{p-primary components of homotopy groups III. Stable groups of the sphere}, Mem. College Sci. Univ. Kyoto Ser. A Math. \textbf{31}, Number 3 (1958), 191--210.
\bibitem{wra} D.J. Wrazidlo, \emph{Fold maps and positive topological quantum field theories}, Dissertation, Heidelberg (2017), \url{http://nbn-resolving.de/urn:nbn:de:bsz:16-heidok-232530}.
\end{thebibliography}

\end{document}